\documentclass[11pt,leqno]{article}
\usepackage{amsmath,amssymb,amsthm,latexsym}
\usepackage{indentfirst,xypic}
\usepackage{amsmath}
\newtheorem{Theorem}{\bf \large Theorem}[section]
\newtheorem{PROPOSITION}[Theorem]{\bf \large Proposition}

\newtheorem{Definition}[Theorem]{\bf \large Definition}
\newtheorem{Remark}[Theorem]{\bf \large Remark}
\newtheorem{ex}[Theorem]{\bf \large Example}
\newtheorem{lemma}[Theorem]{\bf\large Lemma}

    \makeatletter
    
    \newcommand{\Rmnum}[1]{\expandafter\@slowromancap\romannumeral #1@}
    \makeatother

\title{\textbf{Deformation of Hypersurfaces Preserving the M\"{o}bius Metric and a Reduction Theorem}}
\author {\normalsize {Tongzhu Li$^1$~~~Xiang Ma$^2$~~~Changping Wang$^3$}\\
\small{$^1$Department of Mathematics, Beijing Institute of
Technology, Beijing, China.}\\
\small{$^2$LMAM, School of Mathematical sciences, Peking
University, Beijing, China.}\\
\small{$^3$College of mathematics and computer Science, Fujian Normal University, Fuzhou, China}\\
\small{E-mail: $^1$litz@bit.edu.cn; $^2$maxiang@math.pku.edu.cn;  $^3$cpwang@fjnu.edu.cn.}}
\date{}

\begin{document}
\maketitle\footnotetext [1]{T. Z. Li and C. P. Wang is supported by the grant No.
11171004 of NSFC; } \footnotetext [2]{X. Ma is partially supported
by the grant No. 11171004 and No. 10901006 of NSFC. }
\begin{abstract}
A hypersurface without umbilics in the $(n+1)$-dimensional
Euclidean space $f: M^n\rightarrow R^{n+1}$ is known to be
determined by the M\"{o}bius metric $g$ and the M\"{o}bius second
fundamental form $B$ up to a M\"{o}bius transformation when $n\geq 3$.
In this paper we consider M\"{o}bius rigidity for hypersurfaces and deformations of a hypersurface
preserving the M\"obius metric in the high dimensional case $n\geq 4$.
When the highest multiplicity of principal curvatures is less than
$n-2$, the hypersurface is M\"obius rigid. When the multiplicities of all principal curvatures are constant, deformable hypersurfaces and the possible deformations are also classified completely.
In addition, we establish a Reduction Theorem characterizing the classical construction of cylinders, cones, and rotational hypersurfaces,
which helps to find all the non-trivial deformable examples in our classification with wider application in the future.
\end{abstract}

\section{Introduction}
In submanifold theory a fundamental problem is to investigate which
data are sufficient to determine a submanifold $M$ up to the action of a
certain transformation group $G$ on the ambient space.
The \emph{deformable case} means that there exists non-congruent
(depending on $G$) immersions
with the same given invariants at corresponding points,
and such different immersions are called \emph{deformations} to each other.
In contrast,
the \emph{rigid case} indicates that such deformations do not exist
(or just be congruent by the action of $G$).
In this paper we consider deformations of hypersurfaces $M^n$
preserving the so-called M\"obius metric in the framework of M\"obius geometry
($G$ is the M\"obius transformation group acting on $R^{n+1}\cup\{\infty\}$).

As a background let us review some classical results. It is known that a generic
immersed surface in Euclidean three-space $u:M^2\rightarrow R^3$ is
determined, up to a rigid motion of $R^3$, by its induced metric $I$ and
mean curvature function $H$. All exceptional immersions are called
\emph{Bonnet surfaces}, which were classified by
Bonnet\cite{B}, Cartan\cite{c4} and Chern\cite{ch} into three
distinct classes:

1) CMC (constant mean curvature) surfaces with a $1$-parameter deformations preserving $I$ and $H$ (known as \emph{the associated family});

2) Not CMC and admits a continuous $1$-parameter deformations;

3) Surfaces that admit exactly one such deformation.

In either case, two Bonnet surfaces forming deformation to each other is called a \emph{Bonnet pair}. These notions are directly generalized to other space forms
$S^3$ and $H^3$. See \cite{bobenko, kamberov,lawson, springborn} for recent works on this topic.

For a  hypersurface $f:M^n\rightarrow R^{n+1}(n\geq 3)$,
the well-known Beez-Killing rigidity theorem says that
$f$ is isometrically rigid if the rank of its second fundamental form
(i.e. the number of non-zero principal curvatures)
is greater than or equal to $3$ everywhere. Compared to surface case
this is a stronger rigidity result, mainly due to the Gauss equations
which forms an over-determined system when there are many non-zero principal curvatures.

On the other hand, all isometrically deformable hypersurfaces have rank $2$ or less.
They are locally classified by Sbrana \cite{sb} and Cartan \cite{c}.
According to their results there are four classes of them.
The first two classes (surface-like and ruled) are highly deformable.
The third class admits precisely a continuous $1$-parameter family
of deformations, and the fourth class has a unique deformation.

In M\"{o}bius geometry, let $f,\bar{f}: M^n\rightarrow
R^{n+1}$ be two hypersurfaces in the $(n+1)$-dimensional Euclidean
space $R^{n+1}$. We say $f$ is M\"{o}bius equivalent to
$\bar{f}$ (or $f$ is M\"{o}bius congruent to $\bar{f}$) if there
exists a M\"{o}bius transformation $\Psi$ such that $f=\Psi\circ\bar{f}$.
It is natural to consider deformations preserving certain conformal
invariants. In \cite{c1} Cartan considered the problem of \emph{conformal deformation},
i.e. deformation of any given hypersurface preserving the conformal class
of the induced metric.
Cartan has given the following conformal rigidity result:
\begin{Theorem}\cite{c1}
A hypersurface $f:M^n\rightarrow R^{n+1}~~(n\geq 5)$ is conformally
rigid if each principal curvature has multiplicity less than $n-2$
everywhere.
\end{Theorem}
In \cite{m} do Carmo and Dajczer generalized Cartan's rigidity theorem to
submanifolds of dimension $n\geq 5$. Note that the multiplicity of a principal curvature is M\"obius invariant.
When the highest multiplicity is $n$ or $n-1$, it is the
conformally flat case well-known to be highly deformable.
When $n\geq 5$ and the highest multiplicity is $n-2$,
Cartan \cite{c1} gave a quite similar classification
of conformally deformable hypersurfaces into four cases:

I) Surface-like hypersurfaces (which are cylinders, cones or revolution hypersurfaces over surfaces in 3-dim space forms);

II) Conformally ruled hypersurfaces;

III) One of those having a continuous 1-parameter family of deformations;

IV) One of those that admits a unique deformation.

In \cite{d1} and \cite{d2} Dajczer et.al.
gave a modern account of Sbrana and Cartan's classification. Following Dajczer,
we call such conformally deformable hypersurfaces as \emph{Cartan hypersurfaces}
of class I, II, III, and IV.

We observe that in the conformal class of a given immersed hypersurface in $R^{n+1}$
there is a distinguished metric called the \emph{M\"obius metric $g$}.
Together with the \emph{M\"{o}bius second fundamental form $B$} they
form a complete system of invariants in M\"obius geometry (see \cite{w} or Theorem~\ref{fundthe} in this paper).
Based on our experience, the deformation preserving the M\"{o}bius metric $g$ seems to be a natural and new topic.

\begin{Definition}
A hypersurface $f:M^n\rightarrow R^{n+1}$
is said to be \emph{M\"{o}bius rigid} if any other immersion
$\bar{f}:M^n\rightarrow R^{n+1}$ sharing the same M\"{o}bius metric $g$ as $f$, is M\"{o}bius equivalent to $f$.
An immersion $\bar{f}:M^n\rightarrow R^{n+1}$
is said to be a \emph{M\"{o}bius deformation} of $f$
if they induce the same M\"{o}bius metric $g$ at corresponding points
and $\bar{f}(M)$ is not congruent to $f(M)$ up to any M\"obius transformation.
\end{Definition}
We obtain the following M\"obius Rigidity Theorem.
\begin{Theorem}\label{main1}
Let $f: M^n\rightarrow R^{n+1}~~(n\geq 4)$ be a hypersurface in the
$(n+1)$-dimensional Euclidean space.
If every principal curvature of $f$ has multiplicity less
than $n-2$ everywhere, then $f$ is M\"{o}bius rigid.
\end{Theorem}
\begin{Remark}\label{rem1}
Compared with Cartan's notion before, a conformally rigid
hypersurface $f: M^n\rightarrow R^{n+1}$ is M\"obius rigid,
but the converse may not be true. On the other hand, if $f$ is
M\"obius deformable with deformation $\bar{f}$, they are also
conformal deformations to each other, but the converse may not be true. Thus when $n\ge 5$ our rigidity theorem is a corollary of Cartan's conformal rigidity result.

On the other hand, Cartan treated the special dimensions
$n=4,3$ in \cite{c2, c3}. In particular, in \cite{c2}
Cartan has shown that, for $n=4$,
there exist hypersurfaces $f,\bar{f}:M^4\rightarrow R^5$ that have
four distinct principal curvatures at each point $p\in M^4$
and are conformal deformations to each other.
In contrast, our M\"obius rigidity result as above
still holds true for dimension $n=4$. Because of this
interesting difference and for the purpose of self-containedness
we give a proof to Theorem~\ref{main1} in Section~7.
\end{Remark}
The main result of this paper is the following
classification theorem of all M\"obius deformable hypersurfaces.
\begin{Theorem}\label{main2}
Let $f: M^n\rightarrow R^{n+1}~~(n\geq 4)$ be an umbilic free hypersurface in the
$(n+1)$-dimensional Euclidean space, whose principal curvatures have constant multiplicities. Suppose $f$ is M\"{o}bius deformable.

$1)$ When one principal curvature of $f$ has multiplicity $n-1$
everywhere, this deformable $f$ must have constant M\"{o}bius sectional curvature.
They are either cones, cylinders or rotational hypersurfaces
over the so-called \emph{curvature-spirals} in 2-dimensional space-forms.
(See \cite{guo} for the classification or Section~4 for an independent proof.)

$2)$ When one principal curvature of $f$ has multiplicity $n-2$
everywhere, locally $f$ is M\"obius equivalent to
either of the three classes below:

(a) $f(M^n)\subset L^2\times R^{n-2}$, where $L^2$ is a Bonnet
surface in $R^3$;

(b) $f(M^n)\subset CL^2\times R^{n-3}$, where $CL^2\subset R^4$ is a
cone over $L^2\subset S^3$, and $L^2$ is a Bonnet surface in $S^3$;

(c) $f(M^n)$ is a rotational hypersurface over $L^2\subset R^3_+$,
where $L^2$ is a Bonnet surface in the hyperbolic half space model $R^3_+$.

Moreover, the M\"obius deformation to any of them belongs to the same class and comes from the deformation of the corresponding Bonnet surface $L^2$.
\end{Theorem}
\begin{Remark}\label{rek1}
The hypothesis that the principal curvatures have constant multiplicities is necessary in this paper, because we need smooth frame of principal vectors.
But the hypothesis is weak, for there always exists an open dense subset $U$ of $M^n$ on which the multiplicities of the principal curvatures are locally constant (see \cite{reck}).

However,  the hypothesis is not necessary only in  Theorem \ref{main1}.  For dimension $n=4$, the condition that any principal curvature has multiplicity less than $n-2$ means that the principal curvatures have constant multiplicities. For dimension $n\geq 5$,  the frame of principal vectors used in Proposition \ref{prop62} is pointwise , so we do not need smooth frame of principal vector fields.
\end{Remark}
\begin{Remark}\label{rem2}
According to our classification result, among Cartan hypersurfaces
\cite{d2}, only the first class (surface-like hypersurfaces) may share
the same M\"{o}bius metric with their conformal deformations.
The other three classes of conformally deformable hypersurfaces
are M\"{o}bius rigid in our sense.
\end{Remark}
\begin{Remark}\label{rem3}
In the definition above, it is noteworthy that
the \emph{non-congruence} between $\bar{f}(M), f(M)$
(the images) is stronger than the
\emph{non-congruence} between $\bar{f}, f$
(the mappings),
because the same hypersurface $f(M)\subset R^n$ might be
given different parameterizations $f$ and $\bar{f}$
which are NOT M\"obius equivalent.
In other words there might exist an (isometrical) diffeomorphism
$\phi: M^n\to M^n$ and a M\"obius
transformation $\Psi:R^{n+1}\cup\{\infty\}\to R^{n+1}\cup\{\infty\}$
such that the following diagram commutes:
\[
\xymatrix{
  M^n \ar[d]_{f} \ar[r]^{\psi}
                & M^n \ar[d]^{\bar{f}}  \\
  R^{n+1}\cup\{\infty\} \ar[r]_{\Psi}
                & R^{n+1}\cup\{\infty\}             }
                \]
A typical example is the
\emph{M\"obius isoparametric hypersurface}
(see \cite{lw, hu} or Section~9 for the definition) with
three distinct constant M\"{o}bius principal curvatures
\[
\sqrt{\frac{n-1}{2n}},-\sqrt{\frac{n-1}{2n}},0,
\cdots,0.
\]
It is part of the cone over the
Cartan minimal isoparametric hypersurface
$y:N^3\to S^4(1)\hookrightarrow R^5\subset R^{n+1}$
 with three distinct principal curvatures.
This $N^3$ is a tube of a specific constant radius over
the Veronese embedding $RP^2\hookrightarrow S^4$.
It is well-known that its induced metric has a
4-dimensional isometry group whose elements do NOT
preserve the principal distributions in general.
Any such isometry $\phi$ extends to an isometry of the cone
(with respect to its M\"obius metric $g$) which is surely NOT
a M\"obius transformation of the ambient space.
Any possible deformation $\bar{f}$ to the cone $f$ preserving
M\"obius metric $g$ arises in this way, hence
is excluded from our notion (as well as the classification list)
 of M\"obius deformable hypersurfaces. See the discussion of this example in Section~9.
\end{Remark}
\begin{Remark}\label{rem4}
For a hypersurface $f:M\to R^{n+1}$ of constant M\"obius
curvature $c$, generally we can map any neighborhood of a
given point $p\in M$ to a neighborhood of another point
$q\in M$ by an isometry (of $(M,g)$) which is not induced from
a M\"obius transformation of the ambient space.
This is because any such hypersurface is conformally flat with
a specific principal direction which is not preserved by
a generic isometry of $(M,g)$. So they provide
the first class of deformable hypersurfaces.

Circular cylinder and spiral cylinder (constructed from
a circle or a logarithmic spiral, respectively)
belong to this class, yet they are different.
Each of them is homogeneous, namely invariant under a subgroup
of the M\"{o}bius group (of dimension at least $n$) which
acts transitively on $M^n$). On the other hand each of them
have a bigger isometry group (with respect to $(M^n,g)$)
which generally are not induced from M\"obius transformations.
So they resemble Cartan's example in the previous remark.
Yet these two hypersurfaces still have non-trivial deformations. See final remarks in Section~4.
\end{Remark}
\begin{Remark}\label{rem5}
Some comments on low dimensional case $n=3$ or $2$. We do not have any M\"{o}bius rigidity result because our algebraic theorem \ref{th1} fails in this case (see Remark~\ref{rem-dim3}).
But the construction of M\"{o}bius deformable hypersurfaces in Section~3 is still valid for $n=3$.

When $n=2$, generally a surface with a given M\"obius metric
is highly deformable. So we would consider deformation problems
under stronger restrictions. We just mention that any
Willmore surface admits a one-parameter associated family
of Willmore surfaces endowed with the same M\"obius metric.
For more on related topics see \cite{fujioka}.
\end{Remark}
\begin{Remark}\label{rem6}
It is very interesting that the non-trivial deformable examples
all arise from the classical construction of
cylinders, cones or rotational hypersurfaces over a given
hypersurface in a low-dimensional Euclidean subspace,
sphere or hyperbolic half-space, respectively.
Such constructions appeared many times in various contexts
and problems in M\"obius geometry and Lie sphere geometry.
We find that such examples have a nice characterization
(Theorem~\ref{retheorem}) in terms of its M\"obius invariants
introduced by the third author in \cite{w}.
We believe that this Reduction Theorem is a valuable tool
in simplifying discussions of many similar problems.
\end{Remark}
We organize the paper as follows.
In Section~2, we introduce M\"{o}bius invariants and
the M\"{o}bius congruence theorem for hypersurfaces
in $R^{n+1}~~(n\geq 3)$. Examples of M\"{o}bius deformable
hypersurfaces are given in Section~3 and 4
(in particular, Section~4 gives a new proof to the
classification theorem of hypersurfaces with constant
M\"obius sectional curvature). These examples are characterized
by our Reduction Theorems \ref{retheorem} (used in Section~9)
and \ref{retheorem2} (used in Section~4) proved in Section~5.

After these preparations,
as a purely algebraic consequence of the Gauss equation
we show in Section~6 that the (M\"obius) second fundamental
forms of $f$ and its deformation $\bar{f}$
could be diagonalized almost simultaneously.
Then we investigate our problem case by case.
When the highest multiplicity is less than $n-2$ we establish
the rigidity result (Theorem~\ref{main1}) in Section~7.
Section~8 treats the conformally flat case (i.e.
the highest multiplicity is equal to $n-1$) where we show such
deformable examples must have constant M\"obius curvature,
which have been classified in Section~4.
In Section~9 all deformable hypersurfaces with one principal
curvature of multiplicity $n-2$ are proved to be reducible to
cylinders, cones or rotational hypersurfaces using the
Reduction Theorem in Section~5. This finishes the proof
to the Main Theorem~\ref{main2}.

\section{M\"{o}bius invariants for hypersurfaces in $R^{n+1}$}
In this section we briefly review the theory of hypersurfaces
in M\"obius geometry. For details we refer to $\cite{w},\cite{w2}$.

Let $R^{n+3}_1$ be the Lorentz space, i.e., $R^{n+3}$ with inner
 product $<\cdot,\cdot>$ defined by
\[<x,y>=-x_0y_0+x_1y_1+\cdots+x_{n+2}y_{n+2},\] for
$x=(x_0,x_1,\cdots,x_{n+2}), y=(y_0,y_1,\cdots,y_{n+2})\in R^{n+3}$.

Let $f:M^{n}\rightarrow R^{n+1}$ be a hypersurface without umbilics and
assume that $\{e_i\}$ is an orthonormal basis with respect to the
induced metric $I=df\cdot df$ with $\{\theta_i\}$ the dual basis.
Let $II=\sum_{ij}h_{ij}\theta_i\theta_j$ and
$H=\sum_i\frac{h_{ii}}{n}$ be the second fundamental form and the
mean curvature of $f$, respectively. We define the M\"{o}bius
position vector $Y: M^n\rightarrow R^{n+3}_1$ of $f$ by
$$Y=\rho\left(\frac{1+|f|^2}{2},\frac{1-|f|^2}{2},f\right)~,
~~\rho^2=\frac{n}{n-1}(|II|^2-nH^2).$$
\begin{Theorem}\cite{w}
Two hypersurfaces $f,\bar{f}: M^n\rightarrow R^{n+1}$ are M\"{o}bius
equivalent if and only if there exists $T$ in the Lorentz group
$O(n+2,1)$ in $R^{n+3}_1$ such that $\bar{Y}=YT.$
\end{Theorem}
It follows immediately from Theorem 2.1 that
$$g=<dY,dY>=\rho^2df\cdot df$$
is a M\"{o}bius invariant, called the M\"{o}bius metric of $f$.

Let $\Delta$ be the Laplacian with respect to $g$. Define
$$N=-\frac{1}{n}\Delta Y-\frac{1}{2n^2}<\Delta Y,\Delta Y>Y,$$
which satisfies
$$<Y,Y>=0=<N,N>, ~~<N,Y>=1~.$$

Let $\{E_1,\cdots,E_n\}$ be a local orthonormal basis for $(M^n,g)$
with dual basis $\{\omega_1,\cdots,\omega_n\}$. Write
$Y_i=E_i(Y)$. Then we have
$$<Y_i,Y>=<Y_i,N>=0, ~<Y_i,Y_j>=\delta_{ij}, ~~1\leq i,j\leq n.$$
Let $\xi$ be the mean curvature sphere of $f$ written as
$$\xi=\left(\frac{1+|f|^2}{2}H+f\cdot e_{n+1},\frac{1-|f|^2}{2}H-f\cdot e_{n+1},Hf+e_{n+1}\right),$$
where $e_{n+1}$ is the unit normal vector field of $f$ in $R^{n+1}$.

Then $\{Y,N,Y_1,\cdots,Y_n,\xi\}$ forms a moving frame in
$R^{n+3}_1$ along $M^n$. We will use the following range of indices
in this section: $1\leq i,j,k\leq n$. We can write the structure
equations as following:
\begin{eqnarray*}
&&dY=\sum_iY_i\omega_i,\\
&&dN=\sum_{ij}A_{ij}\omega_iY_j+\sum_iC_i\omega_i\xi,\\
&&dY_i=-\sum_jA_{ij}\omega_jY-\omega_iN+\sum_j\omega_{ij}Y_j+\sum_jB_{ij}\omega_j\xi,\\
&&d\xi=-\sum_iC_i\omega_iY-\sum_{ij}\omega_iB_{ij}Y_j,
\end{eqnarray*}
where $\omega_{ij}$ is the connection form of the M\"{o}bius metric
$g$ and $\omega_{ij}+\omega_{ji}=0$. The tensors
$$
{\bf A}=\sum_{ij}A_{ij}\omega_i\otimes\omega_j,~~
{\bf B}=\sum_{ij}B_{ij}\omega_i\otimes\omega_j,~~
\Phi=\sum_iC_i\omega_i$$ are called the
Blaschke tensor, the M\"{o}bius second
fundamental form and the M\"{o}bius form of $f$, respectively.
The covariant derivative of
$C_i, A_{ij}, B_{ij}$ are defined by
\begin{eqnarray*}
&&\sum_jC_{i,j}\omega_j=dC_i+\sum_jC_j\omega_{ji},\\
&&\sum_kA_{ij,k}\omega_k=dA_{ij}+\sum_kA_{ik}\omega_{kj}+\sum_kA_{kj}\omega_{ki},\\
&&\sum_kB_{ij,k}\omega_k=dB_{ij}+\sum_kB_{ik}\omega_{kj}+\sum_kB_{kj}\omega_{ki}.
\end{eqnarray*}
The integrability conditions for the structure equations are given
by
\begin{eqnarray}
&&A_{ij,k}-A_{ik,j}=B_{ik}C_j-B_{ij}C_k,\label{equa1}\\
&&C_{i,j}-C_{j,i}=\sum_k(B_{ik}A_{kj}-B_{jk}A_{ki}),\label{equa2}\\
&&B_{ij,k}-B_{ik,j}=\delta_{ij}C_k-\delta_{ik}C_j,\label{equa3}\\
&&R_{ijkl}=B_{ik}B_{jl}-B_{il}B_{jk}
+\delta_{ik}A_{jl}+\delta_{jl}A_{ik}
-\delta_{il}A_{jk}-\delta_{jk}A_{il},\label{equa4}\\
&&R_{ij}:=\sum_kR_{ikjk}=-\sum_kB_{ik}B_{kj}+(tr{\bf
A})\delta_{ij}+(n-2)A_{ij},\label{equa5}\\
&&\sum_iB_{ii}=0, \sum_{ij}(B_{ij})^2=\frac{n-1}{n}, tr{\bf
A}=\sum_iA_{ii}=\frac{1}{2n}(1+n^2\kappa),\label{equa6}
\end{eqnarray}
where $R_{ijkl}$ denote the curvature tensor of $g$,
$\kappa=\frac{1}{n(n-1)}\sum_{ij}R_{ijij}$ is its normalized
M\"{o}bius scalar curvature. We know that all coefficients in the
structure equations are determined by $\{g, {\bf B}\}$ and we have
\begin{Theorem}$\cite{w}$\label{fundthe}
Two hypersurfaces $f: M^n\rightarrow R^{n+1}$ and $\bar{f}:
M^n\rightarrow R^{n+1} (n\geq 3)$ are M\"{o}bius equivalent if and
only if there exists a diffeomorphism $\varphi: M^n\rightarrow M^n$
which preserves the M\"{o}bius metric and the M\"{o}bius second
fundamental form.
\end{Theorem}
The second covariant derivative of $B_{ij}$ are defined by
$$dB_{ij,k}+\sum_mB_{mj,k}\omega_{mi}+\sum_mB_{im,k}\omega_{mj}+\sum_mB_{ij,m}\omega_{mk}=\sum_mB_{ij,km}\omega_m.$$
We have the following Ricci identities
$$B_{ij,kl}-B_{ij,lk}=\sum_mB_{mj}R_{mikl}+\sum_mB_{im}R_{mjkl}.$$
Coefficients of M\"{o}bius invariants and
Euclidean invariants are related by \cite{w2}
\begin{equation}\label{re2}
\begin{split}
B_{ij}&=\rho^{-1}(h_{ij}-H\delta_{ij}),\\
C_i&=-\rho^{-2}[e_i(H)+\sum_j(h_{ij}-H\delta_{ij})e_j(\log\rho)],\\
A_{ij}&=-\rho^{-2}[Hess_{ij}(\log\rho)-e_i(\log\rho)e_j(\log\rho)-Hh_{ij}]\\
&-\frac{1}{2}\rho^{-2}(|\nabla \log\rho|^2+H^2)\delta_{ij},
\end{split}
\end{equation}
where $Hess_{ij}$ and $\nabla$ are the Hessian matrix and the
gradient with respect to $I=df\cdot df$. Then
$$
A=\rho^2\sum_{ij}A_{ij}\theta_i\otimes\theta_j,~
B=\rho^2\sum_{ij}B_{ij}\theta_i\otimes\theta_j,~\Phi=\rho\sum_iC_i\theta_i.$$
We call eigenvalues
of $(B_{ij})$ as \emph{M\"{o}bius principal curvatures} of $f$. Clearly the
number of distinct M\"{o}bius principal curvatures is the same as
that of its distinct Euclidean principal curvatures.

Let $k_1,\cdots,k_n$ be the principal curvatures of $f$, and
$\{\lambda_1,\cdots,\lambda_n\}$ the corresponding M\"{o}bius
principal curvatures, then the curvature sphere of principal
curvature $k_i$ is
$$\xi_i=\lambda_iY+\xi=\left(\frac{1+|f|^2}{2}k_i+f\cdot e_{n+1},\frac{1-|f|^2}{2}k_i-f\cdot e_{n+1},k_if+e_{n+1}\right).$$
Note that $k_i=0$ if, and only if,
$$<\xi_i,(1,-1,0,\cdots,0)>=0.$$
This means that the curvature sphere of principal curvature $k_i$
is a hyperplane in $R^{n+1}$.

\section{Examples of M\"{o}bius deformable hypersurfaces}
This section describes the construction of M\"obius deformable hypersurfaces $M^n$ whose highest multiplicity of principal curvatures is $n-2$.
\begin{ex}\label{ex31}
Let $u: L^m\longrightarrow R^{m+1}$ be an immersed hypersurface. We define
\emph{the cylinder over $u$} in $R^{n+1}$ as
\[
f=(u,id):L^m\times R^{n-m}\longrightarrow
R^{m+1} \times R^{n-m}=R^{n+1},
\]
where $id:R^{n-m}\longrightarrow R^{n-m}$ is the identity map.
\end{ex}
\begin{PROPOSITION}\label{31}
Let $u,\bar{u}: L^2\longrightarrow R^3$ be a Bonnet pair. Then the cylinders
$f=(u,id):L^2\times R^{n-2}\longrightarrow R^{n+1}$ and
$\bar{f}=(\bar{u},id)$ are M\"{o}bius deformations to each other.
\end{PROPOSITION}
\begin{proof} Let $\eta$ be the unit
normal vector of surface $u$. Then $e_{n+1}=(\eta,\vec{0})\in
R^{n+1}$ is the unit normal vector of hypersurface $f$. The first
fundamental form $I$ and the second fundamental form $II$ of
hypersurface $f$ are given by
\begin{equation}\label{II31}
I=I_u+I_{R^{n-2}}, \;\; II=II_u,
\end{equation}
where $I_u,II_u$ are the first and second fundamental forms of $u$,
respectively, and $I_{R^{n-2}}$ denotes the standard metric of $R^{n-2}$.
Let $k_1,k_2$ be principal curvatures of surface $u$. The principal curvatures of hypersurface $f$ are obviously
$$k_1,k_2,0,\cdots,0.$$
The M\"{o}bius metric $g$ of hypersurface $f$ is
\begin{equation}\label{eq-g1}
g=\rho^2I=\frac{n}{n-1}(|II|^2-nH^2)I
=\left(4H_u^2-\frac{2n}{n-1}K_u\right)(I_u+I_{R^{n-2}}),
\end{equation}
where $H_u,K_u$ are the mean curvature of $u$ and
Gauss curvature of $u$, respectively.
Since $\bar{u}:L^2\longrightarrow R^3$ share the same metric
$I_u$ and mean curvature $H_u$ as $u$, the cylinder
$\bar{f}=(\bar{u},id):L^2\times R^{n-2}\longrightarrow R^{n+1}$
share the same factor $\rho$ and M\"obius metric, i.e.
\[
g=\bar{g}.
\]
Note that the correspondence between the Bonnet pair $u,\bar{u}$
preserves the principal curvatures,
yet NOT the principal directions.
By \eqref{II31} this is also true between $f,\bar{f}$.
So we conclude that
$\bar{f}$ is a non-trivial M\"obius deformation to $f$. This completes the proof to Proposition \ref{31}.
\end{proof}
\begin{ex}\label{ex32}
Let $u:L^m\longrightarrow S^{m+1}\subset R^{m+2}$ be an
immersed hypersurface.
We define \emph{the cone over $u$} in $R^{n+1}$ as
\begin{equation*}
\begin{split}
&f:L^m\times R^+\times R^{n-m-1}\longrightarrow R^{n+1},\\
&~~~~~~f(u,t,y)=(tu,y),
\end{split}
\end{equation*}
\end{ex}
\begin{PROPOSITION}\label{32}
Let $u,\bar{u}:L^2\longrightarrow S^3$ be a Bonnet pair in the
standard 3-sphere. Then the cone hypersurfaces $f:L^2\times
R^+\times R^{n-3}\longrightarrow R^{n+1}$ and $\bar{f}$
over them are M\"{o}bius deformations to each other.
\end{PROPOSITION}
\begin{proof}
The first and second
fundamental forms of hypersurface $f$ are, respectively,
\[
I=t^2I_u+I_{R^{n-2}}, \;\; II=t~II_u,
\]
where $I_u,II_u,I_{R^{n-2}}$ are understood as before.
Let $k_1,k_2$ be principal curvatures of surface $u$.
The principal curvatures of hypersurface $f$ are
\[
\frac{1}{t}k_1,\frac{1}{t}k_2,0,\cdots,0.
\]
Thus the M\"{o}bius metric $g$ of hypersurface $f$ is
\begin{equation}\label{eq-g2}
\begin{split}
g=\rho^2I
&=\frac{1}{t^2}\left[4H_u^2-\frac{2n}{n-1}(K_u-1)\right]
(t^2I_u+I_{R^{n-2}})\\
&=\left[4H_u^2-\frac{2n}{n-1}(K_u-1)\right]
(I_u+I_{H^{n-2}}),
\end{split}
\end{equation}
where $H_u,K_u$ are the mean curvature and Gauss curvature of
$u$, respectively, $I_{H^{n-2}}$ is the standard hyperbolic of $R^{n-2}_+=R^+\times R^{n-3}$.
Since $\bar{u}:L^2\longrightarrow S^3$ share the same metric
$I_u$ and mean curvature $H_u$ as $u$, the cone over $\bar{u}$
$\bar{f}:L^2\times R^+\times R^{n-3}\longrightarrow R^{n+1}$
share the same M\"obius metric, i.e.
\[
g=\bar{g}.
\]
By the same reason in the proof to Proposition~\ref{31},
we know that their principal directions do NOT correspond.
So they are genuine deformations to each other. This completes the proof to Proposition \ref{32}.
\end{proof}

\begin{ex}\label{ex33}
Let $R^{m+1}_+=\{(x_1,\cdots,x_m,x_{m+1})\in R^{m+1}|x_{m+1}>0\}$
be the upper half-space endowed with the standard
hyperbolic metric
\[
ds^2=\frac{1}{x_{m+1}^2}\sum_{i=1}^m dx_i^2.
\]
Let $u=(x_1,\cdots,x_{m+1}):M^m\longrightarrow R^{m+1}_+$ be
an immersed hypersurface. We define
\emph{rotational hypersurface over $u$} in $R^{n+1}$ as
\begin{equation*}
\begin{split}
&f:L^m\times S^{n-m}\longrightarrow R^{n+1},\\
&f(x_1,\cdots,x_{m+1},\phi)=(x_1,\cdots,x_m,x_{m+1}\phi),
\end{split}
\end{equation*}
where $\phi:S^{n-m}\longrightarrow R^{n-m+1}$ is the standard sphere.
\end{ex}
\begin{PROPOSITION}\label{33}
Let $u,\bar{u}:L^2\longrightarrow R^3_+$ be
a Bonnet pair in the hyperbolic 3-space.
Then the rotational hypersurfaces
$f=(x_1,x_2,x_3\phi):L^2\times S^{n-2}\longrightarrow R^{n+1}$
and $\bar{f}=(\bar{x}_1,\bar{x}_2,\bar{x}_3\phi)$ are
M\"{o}bius deformations to each other.
\end{PROPOSITION}
\begin{proof} Let $R^4_1$ be the
Lorentz space with inner product
\[
<y,y>=-y_1^2+y_2^2+y_3^2+y_4^2,\;\; y=(y_1,y_2,y_3,y_4).
\]
Let $H^3=\{y\in R^4_1|<y,y>=-1,y_1>0\}$ be the hyperbolic space.
Introduce isometry $\tau:R^3_+\longrightarrow H^3$ as below:
\[
\tau(x_1,x_2,x_3)=\left(\frac{1+x_1^2+x_2^2+x_3^2}{2x_3},
\frac{1-x_1^2-x_2^2-x_3^2}{2x_3},\frac{x_1}{x_3},
\frac{x_2}{x_3}\right).
\]
The inverse $\tau^{-1}:H^3\longrightarrow R^3_+$ is
$\tau^{-1}(y_1,y_2,y_3,y_4)=(\frac{y_3}{y_1+y_2},
\frac{y_4}{y_1+y_2},\frac{1}{y_1+y_2}).$

Let $\eta$ be the unit normal vector of surface $u$ in $R^3_+$.
Write
$\eta=(\eta_1,\eta_2,\eta_3).$
Since $\eta$ is the unit normal vector,then
\[
\frac{\eta_1^2+\eta_2^2+\eta_3^2}{x_3^2}=1.
\]
Thus the unit normal vector of hypersurface $f$ in $R^{n+1}$ is
\[
\xi=\frac{1}{x_3}(\eta_1,\eta_2,\eta_3\phi).
\]
The first fundamental form of $u$ is
\[
I_u=\frac{1}{x_3^2}(dx_1\cdot dx_1+dx_2\cdot dx_2+dx_3\cdot dx_3).
\]
The second fundamental form of $u$ is
\[
II_u=-<\tau_*(du),\tau_*(d\eta)>
=\frac{1}{x_3^2}(dx_1\cdot d\eta_1+dx_2\cdot d\eta_2
+dx_3\cdot d\eta_3)-\frac{\eta_3}{x_3}I_u.
\]
Now we can write out the first and the second fundamental forms
of $f$:
\[
I=df\cdot df=x_3^2(I_u+I_{S^{n-2}}),
~~II=x_3II_u-\eta_3I_u-\eta_3I_{s^{n-2}},
\]
where $I_{S^{n-2}}$ is the standard metric of $S^{n-2}$. Let
$k_1,k_2$ be principal curvatures of $u$. Then principal
curvatures of hypersurface $f$ are
\[
\frac{k_1}{x_3}-\frac{\eta_3}{x_3^2}~,~
\frac{k_2}{x_3}-\frac{\eta_3}{x_3^2}~,~\frac{-\eta_3}{x_3^2}~,
\cdots,\frac{-\eta_3}{x_3^2}~.
\]
Thus
\[
\rho^2=\frac{n}{n-1}(|II|^2-nH^2)
=\frac{1}{x^2_3}\left[4H_u^2-\frac{2n}{n-1}(K_u+1)\right],
\]
where $H_u,K_u$ are the mean curvature and Gauss curvature of
$u$, respectively.
So the M\"{o}bius metric of hypersurface $f$ is
\begin{equation}\label{eq-g3}
g=\rho^2I=\left[4H_u^2-\frac{2n}{n-1}(K_u+1)\right]
(I_u+I_{S^{n-2}}).
\end{equation}
Since $u$ and $\bar{u}$ are a pair of Bonnet surfaces,
$H_u=H_{\bar{u}},K_u=K_{\bar{u}},I_u=I_{\bar{u}}$, thus
$\bar{f}=(\bar{x}_1,\bar{x}_2,\bar{x}_3\phi):L^2\times
S^{n-2}\longrightarrow R^{n+1}$,
the rotational hypersurface over $\bar{u}$,
is endowed with the same M\"obius metric $g$.
Similar to previous discussions we know that they are NOT
congruent. This completes the proof to Proposition \ref{33}.
\end{proof}
\begin{Remark}\label{rem-metric}
We note that the M\"obius metric $g$ in these three cases
\eqref{eq-g1}\eqref{eq-g2}\eqref{eq-g3}
could be unified in a single formula:
\begin{equation}\label{eq-g}
g=\left[4H_u^2-\frac{2n}{n-1}(K_u+c)\right](I_u+I_{N^{n-2}(c)}).
\end{equation}
Here $H_u,K_u,I_u$ are the mean curvature, the Gauss curvature
and the first fundamental form of the surface $u:L^2\to N^3(-c)$
in a three dimensional space form of constant curvature $-c$;
$I_{N^{n-2}(c)}$ is the Riemannian metric of a $(n-2)-$dimensional space form of constant curvature $c$.
This will be used in Section~9 to show that
any M\"obius deformation to any example in these
three propositions arises in this way.
In other words, the possible deformations are as many as that of
the corresponding Bonnet surface.
\end{Remark}

\section{Hypersurfaces with constant M\"{o}bius curvature: deformations and classification}
As pointed out in the introduction, hypersurfaces
with constant M\"{o}bius sectional curvature form a new class of
deformable hypersurfaces.
In this section, we list hypersurfaces with constant M\"{o}bius
curvature, i.e., constant sectional curvature with respect to
the M\"{o}bius metric $g$, and compute the M\"{o}bius
invariants. Then we give a new proof to the classification
of such hypersurfaces using a
reduction theorem~\ref{retheorem2} in Section~5.
\begin{ex}\label{ex1}
The cylinder in $R^{n+1}$ over $\gamma(s)\subset R^2$ is defined by
\[f(s,id)=(\gamma(s),id):~I\times R^{n-1}\longrightarrow R^{n+1},\]
where $id:R^{n-1}\longrightarrow R^{n-1}$ is the identity mapping.
\end{ex}
\begin{Remark}
This is exactly Example~\ref{ex31} when $m=1$.
\end{Remark}
The first fundamental form $I$ and the second
fundamental form $II$ of hypersurface $f$ are, respectively,
\[ I=ds^2+I_{R^{n-1}}, \;\; II=\kappa(s) ds^2,\]
where $\kappa(s)$ is the geodesic curvature of
$\gamma\subset R^2$, $s$ is the arc-length parameter,
and $I_{R^{n-1}}$ is the standard Euclidean metric of $R^{n-1}$.
So we have
$(h_{ij})=\operatorname{diag}(\kappa,0,\cdots,0)~,
~H=\frac{\kappa}{n}~,~\rho=\kappa~.$
Thus the M\"{o}bius metric $g$ of hypersurface $f$ is
\[
g=\rho^2 I=\kappa(s)^2(ds^2+I_{R^{n-1}}).
\]
The M\"{o}bius invariants of $f$ under an orthonormal frame
(consisting of principal directions) can be obtained as below
using \eqref{re2}:
\begin{equation}\label{rre12}
\begin{split}
&C_1=-\frac{\kappa_s}{\kappa^2}~, ~ C_2=\cdots=C_n=0,\\
&(B_{ij})=\operatorname{diag}\left(\frac{n-1}{n},\frac{-1}{n},\cdots,\frac{-1}{n}\right),\\
&(A_{ij})=\operatorname{diag}(a_1,a_2,\cdots,a_2),
\end{split}
\end{equation}
where $a_1=-\dfrac{\kappa_{ss}}{\kappa^3}
+\dfrac{3}{2}\dfrac{(\kappa_s)^2}
{\kappa^4}+\dfrac{2n-1}{2n^2}~,~
a_2=-\dfrac{1}{2}\left[\dfrac{(\kappa_s)^2}{\kappa^4}
+\dfrac{1}{n^2}\right].$
\begin{ex}\label{ex2}
The cone in $R^{n+1}$ over $\gamma(s)\subset S^2(1)\subset R^3$ is defined by
\[
f(s,t,id)=(t\gamma(s),id):~I\times R^{+}\times R^{n-2}\longrightarrow R^{n+1},
\]
where $id:R^{n-2}\longrightarrow R^{n-2}$ is identity mapping and
$R^{+}=\{t|~t>0\}$.
\end{ex}
\begin{Remark}
This is exactly Example~\ref{ex32} when $m=1$.
\end{Remark}
The first and second
fundamental forms of hypersurface $f$ are
\[I=t^2ds^2+I_{R^{n-1}}~, \;\; II=t\kappa(s) ds^2.\]
So we have
$(h_{ij})=\operatorname{diag}
\left(\frac{\kappa}{t},0,\cdots,0\right)~,~H=\frac{\kappa}{nt}~,
~\rho=\frac{\kappa}{t}~.$
Thus the M\"{o}bius metric $g$ of hypersurface $f$ is
\[
g=\rho^2I=\frac{\kappa(s)^2}{t^2}\left(t^2ds^2+I_{R^{n-1}}\right)
=\kappa(s)^2(ds^2+I_{H^{n-1}}),
\]
where $I_{H^{n-1}}$ is the standard hyperbolic metric of $H^{n-1}(-1)$.
The M\"{o}bius invariants of $f$ under an orthonormal frame (consisting of principal directions) can be obtained similarly:
\begin{equation}\label{rre13}
\begin{split}
&C_1=-\frac{\kappa_s}{\kappa^2}~,~ C_2=\cdots=C_n=0,\\
&(B_{ij})=\operatorname{diag}\left(\frac{n-1}{n},\frac{-1}{n},\cdots,\frac{-1}{n}\right),\\
&(A_{ij})=\operatorname{diag}(a_1,a_2,\cdots,a_2),
\end{split}
\end{equation}
where $a_1=-\dfrac{\kappa_{ss}}{\kappa^3}
+\dfrac{3}{2}\dfrac{(\kappa_s)^2}
{\kappa^4}+\dfrac{1}{2\kappa^2}+\dfrac{2n-1}{2n^2}~,~
a_2=-\dfrac{1}{2}\left[\dfrac{(\kappa_s)^2}{\kappa^4}
+\dfrac{1}{\kappa^2}+\dfrac{1}{n^2}\right].$
\begin{ex}\label{ex3}
The rotational hypersurface in $R^{n+1}$ over $\gamma(s)\subset R^2_+=\{(x,y)\in R^2|~y>0\}\subset R^3$ is defined by
\[
f(x,y,\theta)=(x,y\theta):~I\times S^{n-1}\longrightarrow R^{n+1},
\]
where $\theta:S^{n-1}\longrightarrow R^{n}$ is the
standard immersion of a round sphere, $R^2_+$ is regarded as
the Poincare half plane with the hyperbolic metric
$ds^2=\frac{1}{y^2}(dx^2+dy^2)$.
\end{ex}
\begin{Remark}
This is exactly Example~\ref{ex33} when $m=1$.
\end{Remark}
In the Poincare half plane, denote the covariant differentiation
of the hyperbolic metric as $D$. Choose orthonormal frames
$e_1=y\frac{\partial}{\partial x},e_2=y\frac{\partial}{\partial y}$. It is easy to find
\[D_{e_1}e_1=e_2~,~D_{e_1}e_2=-e_1~,~D_{e_2}e_1=D_{e_2}e_2=0.\]
For $\gamma(s)=((x(s),y(s))\subset R^2_+$
let $x'$ denote derivative $\partial x/\partial s$ and so on.
Choose the unit tangent vector
$\alpha=\frac{1}{y}(x'(s)e_1+y'(s)e_2)$ and the unit normal
vector $\beta=\frac{1}{y}(-y'(s)e_1+x'(s)e_2)$.
The geodesic curvature is computed via
\[
\kappa(s)=\langle D_\alpha \alpha,\beta\rangle
=\frac{x'y''-x''y'}{y^2}+\frac{x'}{y}.
\]
After these preparation, we see that the rotational hypersurface
$f(x,y,\theta)=(x,y\theta)$ has differential
$df=(x'ds,y'\theta ds+y d\theta)$ and unit normal vector $\eta=\frac{1}{y}(-y',x'\theta).$
Thus the first and second
fundamental forms of hypersurface $f$ are
\[
I=df\cdot df=y^2(ds^2+I_{S^{n-1}})~,
~ II=-df\cdot d\eta=(y\kappa-x')ds^2-x'I_{S^{n-1}},
\]
where $I_{S^{n-1}}$ is the standard metric of $S^{n-1}(1)$.
Thus principal curvatures are \newline
$\frac{\kappa y-x'}{y^2},\frac{-x'}{y^2},
\cdots,\frac{-x'}{y^2}.$
So $\rho=\frac{\kappa}{y}$, and the M\"{o}bius metric of $f$ is
\[g=\rho^2I=\kappa^2(ds^2+I_{S^{n-1}}).\]
The coefficients of M\"{o}bius invariants are:
\begin{equation}\label{rre14}
\begin{split}
&C_1=-\frac{\kappa_s}{\kappa^2}~,~ C_2=\cdots=C_n=0,\\
&(B_{ij})=\operatorname{diag}\left(\frac{n-1}{n},
\frac{-1}{n},\cdots,\frac{-1}{n}\right),\\
&(A_{ij})=\operatorname{diag}(a_1,a_2,\cdots,a_2),
\end{split}
\end{equation}
where $a_1=\dfrac{\kappa_{ss}}{\kappa^3}
-\dfrac{5}{2}\dfrac{(\kappa_s)^2}{\kappa^4}
-\dfrac{1}{2\kappa^2}+\dfrac{2n-1}{2n^2}~,~
a_2=-\dfrac{1}{2}\left[\dfrac{(\kappa_s)^2}{\kappa^4}
-\dfrac{1}{\kappa^2}+\dfrac{1}{n^2}\right].$
\begin{lemma}\label{metric}
The M\"obius metric of those hypersurfaces
in Examples (\ref{ex1}), (\ref{ex2}) and (\ref{ex3}) are of the warped-product form
\begin{equation}\label{eq-metric}
g=\kappa^2(s)\left(ds^2+I_{-\epsilon}^{n-1}\right),
\end{equation}
where $I_{-\epsilon}^{n-1}$ is the metric of $n-1$ dimensional
space form of constant curvature $-\epsilon$.
This metric \eqref{eq-metric} is of constant sectional
curvature $c$ if, and only if, the function $\kappa(s)$ satisfies
\begin{equation}\label{spiral}
\left[\frac{d}{ds}\frac{1}{\kappa}\right]^2
+\epsilon\left[\frac{1}{\kappa}\right]^2=-c.
\end{equation}
\end{lemma}
The proof is an easy exercise and we omit it at here.
\begin{Definition}
We call a curve $\gamma$ \emph{the curvature-spiral} in a $2-$dimensional space form $N^2(\epsilon)=R^2,S^2,H^2$ (of Gauss curvature $\epsilon=0,1,-1$
respectively),
 if its geodesic curvature $\kappa(s)$ is not constant and
 satisfies \eqref{spiral}.
\end{Definition}
 Note that \eqref{spiral} is
 equivalent to the harmonic oscillator equation for the function $\kappa(s)$:
\[
(1/\kappa)''+\epsilon/\kappa=0.
\]
It is easy to see that for fixed $\epsilon, c$ the
solution curve is unique (because $N^2(\epsilon)$ is a
two-point homogeneous space). In particular, when $\epsilon=0$, $N^2(\epsilon)=R^2$,
the corresponding $\gamma$ is a circle or a logarithmic spiral,
 and the cylinder $\gamma\times R^{n-1}$
is called \emph{the circular cylinder} and
\emph{the spiral cylinder} \cite{sulanke}, respectively.

\begin{Theorem}[\cite{guo}]\label{them1}
Let $f:M^n\rightarrow R^{n+1}$ $(n\geq 3)$ be an umbilic free immersed
hypersurface with constant M\"{o}bius curvature $c$. If $n=3$
we assume that $f$ has two distinct principal curvatures.
Then locally $f$ is M\"{o}bius equivalent to one of the following examples:\\
$(i)$ the circular cylinder (where $c=0$) or the spiral cylinder
(where $c<0$);\\
$(ii)$ a cone over a curvature-spiral in a $2$-sphere (where $c<0$); \\
$(iii)$ a rotation hypersurface over a curvature-spiral in a
hyperbolic $2$-plane (the constant curvature $c$ could be positive, negative or zero).
\end{Theorem}
\begin{proof}
Choose an orthonormal frame with respect to $g$ so that
$(B_{ij})$ is diagonal. According to the following Remark \ref{rek2}, $f$ has
two distinct principal curvatures, one of which is simple.
The assumption of constant curvature for $g$ implies the
Ricci curvature $R_{ij}=0$ for $i\ne j$.
From the integrability equation \eqref{equa5} we deduce that
$(A_{ij})$ is also diagonal.
Thus the second reduction theorem~\ref{retheorem2}
in the next section says that the M\"{o}bius form is closed
and $f$ is reducible.
Invoking Lemma~\ref{metric} we finish the proof.
\end{proof}
\begin{Remark}\label{rek2}
Clearly hypersurfaces with constant M\"{o}bius curvature
are conformally flat. Equivalently, when the dimension $n\geq 4$
there must be a principal curvature of multiplicity $n-1$
everywhere (and the hypersurface is the envelop of
a one-parameter family of $(n-1)$ dimensional spheres).

On the other hand, a $3$-dimensional hypersurface
$f:M^3\rightarrow R^4$ with constant M\"{o}bius sectional
curvature may have three distinct principal curvatures.
We have finished a classification of such examples which will be
published later \cite{lmw}.
\end{Remark}
Let's see for fixed $c$ how many different (global)
examples exist.
If $\epsilon=0,~N^2(\epsilon)=R^2$, without loss of generality
the solution to \eqref{spiral} is written as
\begin{equation}\label{spherical1}
\kappa=1/(\sqrt{-c}s) .~~~~~~~~~~~~~~~~\text{(logarithmic-spiral)}
\end{equation}
When $\epsilon=1,~N^2(\epsilon)=S^2$, without loss of generality
the solution to \eqref{spiral} is written as
\begin{equation}\label{spherical}
\kappa=1/(\sqrt{-c}\sin s) .~~~~~~~~~~~~~~~~\text{(sin-spiral)}
\end{equation}
When $\epsilon=-1,~N^2(\epsilon)=H^2(-1)$, there are three different
possibilities:
\begin{eqnarray}
\kappa=1/(\sqrt{-c}\sinh s),~~~~~~~~&&\text{(sinh-spiral)}\label{sinh}\\
\kappa=1/(\sqrt{c}\cosh s),~~~~~~~~~~&&\text{(cosh-spiral)}\label{cosh}\\
\kappa=e^s.~~~~~~~~~~~~~~~~~~&&\text{(exp-spiral)}\label{exp}
\end{eqnarray}
When $c>0$ we have a unique example (cosh-spiral).
Yet this example is not homogeneous and should not be viewed
as M\"obius rigid according to Remark~\ref{rem4}.

In contrast, for hypersurfaces of M\"obius curvature $c<0$
we have three non-congruent hypersurfaces:
the spiral cylinder, the cone hypersurface, and the rotational hypersurface over
the sinh-spiral. We conclude that either of them
(in particular, the spiral cylinder) is M\"obius deformable.
(See Remark~\ref{rem3} and \ref{rem4}.)

When $c=0$, according to our theorem, there exist two
non-congruent examples: the circular cylinder and
the rotational hypersurfaces over the exp-spiral as
in equation~\eqref{exp}. So either of them is deformable.

\section{The Reduction Theorem}
In this section we establish a criterion in terms of
M\"obius invariants for a hypersurface to be cylinders,
cones and rotational hypersurfaces (Examples
(\ref{ex31})(\ref{ex32})(\ref{ex33})). This is used in the
previous and the final section.
\begin{Theorem}[Reduction Theorem]\label{retheorem}
Let $f:M^n\rightarrow R^{n+1} (n\geq 3)$ be an umbilic free immersed hypersurface, whose principal curvatures have constant multiplicities.
We diagonalize the M\"obius second fundamental form
under an orthonormal frame $\{E_1,E_2,\cdots,E_n\}$
with respect to the M\"obius metric $g$:
\[
B_{ij}=\mathrm{diag}\{\lambda_1,\cdots,\lambda_m,\mu,\cdots,\mu\}.
\]
Assume:

$(1)$~$\lambda_1,\cdots,\lambda_m$ are distinct from $\mu$.

$(2)$~$2\le m\le n-2$. (So the multiplicity of $\mu$ is $n-m$ and $2\le n-m\le n-2$.)

$(3)$~$B_{pq,\alpha}=0,~C_{\alpha}=0,
~~ ~1\le p,q \le m,~m+1\leq \alpha\leq n.$

Then $f$ is M\"{o}bius congruent to one of the examples (\ref{ex31}),(\ref{ex32}) and (\ref{ex33}).
\end{Theorem}
\begin{proof}
Let $\{Y,N,Y_1,\cdots,Y_n,\xi\}$ be a moving frame in
$R^{n+3}_1$ (see Section 2). In the proof below we adopt
the convention on the range of indices as below:
\[
1 \le p,q,r,s,t \le m,~~
m+1 \le \alpha,\beta,\gamma \le n,~~
1 \le i,j,k,l\le n.
\]
Without loss of generality we make a new choice of frame vectors such that
\begin{equation}\label{redua1}
A_{\alpha\beta}=a_{\alpha}\delta_{\alpha\beta}.
\end{equation}
Applying
$dB_{ij}+\sum_kB_{kj}\omega_{ki}+\sum_kB_{ik}\omega_{kj}=\sum_kB_{ij,k}\omega_k$
for off-diagonal element $B_{\alpha\beta}$ ($\alpha\ne\beta$) and
using the fact $B_{\alpha\alpha}=B_{\beta\beta}=\mu,
B_{\alpha\beta}=0$ we get
\begin{equation}\label{5b1}
B_{\alpha\beta,k}=0=B_{k\alpha,\beta},
~~\forall~\alpha\neq\beta,~1\le
k\le n.
\end{equation}
The second equality is by the integrability equation. Since $n-m\ge
2$, we can always choose indices $\alpha\ne\beta$. Then by
integrability equation and the assumption $C_{\beta}=0$ one has
\begin{equation}\label{5b2}
E_{\beta}(\mu)=B_{\alpha\alpha,\beta}=B_{\alpha\beta,\alpha}
+\delta_{\alpha\alpha}C_{\beta}-\delta_{\alpha\beta}C_{\alpha}=C_{\beta}=0,~~\forall
\beta.
\end{equation}
Here $B_{\alpha\beta,\alpha}=0$ due to \eqref{5b1}. Similarly we
have $B_{p\alpha,q}=B_{pq,\alpha}
+\delta_{p\alpha}C_{q}-\delta_{pq}C_{\alpha}=B_{pq,\alpha}$ and
$B_{p\alpha,\alpha}=B_{\alpha\alpha,p}-C_{p} =E_{p}(\mu)-C_{p}$.
Together with the assumption $B_{pq,\alpha}=0$ we summarize that
\begin{equation}\label{5b3}
B_{pq,\alpha}=B_{p\alpha,q}=0,~B_{p\alpha,\alpha}
=E_{p}(\mu)-C_{p},~~\forall~p,q,\alpha.
\end{equation}
Now with the help of \eqref{5b1} and \eqref{5b3} we
compute the covariant derivatives of off-diagonal components
$B_{p\alpha}$ and find
\begin{equation}\label{5b4}
\omega_{p\alpha}=\frac{B_{p\alpha,\alpha}}{\lambda_p-\mu}\omega_{\alpha},
~~\forall~p,\alpha.
\end{equation}
Differentiating once more we obtain the curvature tensor.
Compare the coefficient of the component
$\omega_p\wedge\omega_q$ for any given $p\ne q$
We find that
\[R_{p\alpha pq}=0.\]
(This is the only place where we use the assumption $m\ge 2$, to guarantee that there exist such $p\ne q$).
From the integrability equation~\eqref{equa4} we get
\begin{equation}\label{5b12}
A_{q\alpha}=0~,~~~1\le q\le m,m+1\le \alpha\le n.
\end{equation}
Similarly by comparing the component $\omega_p\wedge\omega_{\alpha}$
we observe that $R_{p\alpha p\alpha}$ is independent of $\alpha$ (here we use
\eqref{5b3}). Equation~\eqref{equa4} yields
$R_{p\alpha p\alpha}=\lambda_p\mu+A_{pp}+A_{\alpha\alpha}$ and
\begin{equation}\label{5a1}
A_{\alpha\alpha}=a,~
\forall ~\alpha~.
\end{equation}

Next we compute the covariant derivatives of tensor $A$ and $C$. By the condition
$C_{\alpha}=0$ and the integrability equation \eqref{equa1}
$A_{ij,k}-A_{ik,j}=B_{ik}C_{j}-B_{ij}C_{k}$,
\begin{equation}\label{5b13}
E_{\alpha}(a)=E_{\alpha}(A_{\beta\beta})=A_{\beta\beta,\alpha}
=A_{\alpha\beta,\beta}=0,~\forall~\alpha\ne\beta.
\end{equation}
As a consequence of \eqref{5b4}
and $dC_i+\sum_kC_k\omega_{ki}=\sum_kC_{i,k}\omega_k$ we get that
\begin{equation}\label{5b14}
E_{\alpha}(C_p)=C_{p,\alpha}=C_{\alpha,p}=0,~\forall~p,\alpha.
\end{equation}

Let's look at the geometric meaning of these results.
From the formula in \eqref{5b4} we know that distributions
\[
D_1\triangleq\mathrm{Span}\{E_p|1\le p\le
m\},~~D_2\triangleq\mathrm{Span}\{E_{\alpha}|m+1\le \alpha\le n\},
\]
are integrable. Any integral submanifold of distribution $D_1$ is a
$m$-dimensional submanifold. On the other hand, along any integral
submanifold of $D_2$ the hypersurface $Y$ is tangent to
\begin{equation}
F\triangleq\mu Y+\xi,
\end{equation}
the principal curvature sphere of multiplicity $n-m$. Using \eqref{5b2}, $E_p(\mu)=B_{\alpha\alpha,p}=B_{p\alpha,\alpha}+C_{p}$ and the
structure equation it is easy to get that
\begin{equation}\label{epf}
E_{\alpha}(F)=0,~E_p(F)=B_{p\alpha,\alpha}Y+(\mu-\lambda_p)Y_p.
\end{equation}
Then principal curvature sphere $F$ induces a $m$-dimensional
submanifold in the de-Sitter space $S^{n+2}_1$
\[
F:\widetilde{M}^m=M^n/L\rightarrow S^{n+2}_1,
\]
where fibers $L$ are integral submanifolds of distribution $D_2$. In
other words, $F$ form a $m$-parameter family of $n$-spheres
enveloped by the hypersurface $Y$.

The next crucial observation is that $F$ is located in a fixed
$(m+2)$-dimensional linear subspace of $R^{n+3}_1$. To show that we
compute the repeated derivatives of $F$, which contains
all information of the envelope $Y$.
Straightforward yet tedious computation shows that the frames of
\begin{equation}\label{eq-v1}
V_1\triangleq\text{Span}\{F,E_1(F),\cdots,E_m(F),P\},
\end{equation}
\[\text{where}~~~~~P\triangleq A_{\alpha\alpha}Y-N+\sum_{p=1}^m\frac{B_{p\alpha,\alpha}}
{(\mu-\lambda_p)^2}E_p(F) +\mu F,\]
satisfy a linear first order PDE system. Hence
these vectors, including $F$ itself, are contained in a fixed
$(m+2)$-dimensional subspace $V_1$ endowed with degenerate, Lorentzian,
or positive definite inner product. This agrees with the geometry of
cylinders, cones, and rotational hypersurfaces (see examples (3.1),(3.3),(3.5)),
where the principal curvature sphere $F$ is orthogonal to a
$(n-m+1)$-parameter family of hyperplanes/hyperspheres.
Moreover, the orthogonal complement $V_1^{\perp}$ of $dim=n-m+1$
contains all $Y_{\alpha},~~ m+1\le\alpha\le n$.

The final fact above inspires us to proceed in an alternative
and easier way. Differentiate any given $Y_{\alpha}$
and modulo components in the subspace
$\text{Span}\{Y_{\gamma},~m+1\le\gamma\le n\}$. By \eqref{redua1}\eqref{5b12}\eqref{5b4} one finds
\begin{eqnarray}
E_i(Y_{\alpha})&=&-A_{\alpha i}Y-\delta_{\alpha i}N
+\sum{_{j}}~\omega_{\alpha j}(E_i)Y_j+B_{\alpha i}\xi \notag \\
&=&\left\{
\begin{array}{ll}
-T~(\text{mod}~ Y_{\gamma}),~\text{when}~i=\alpha~;\\
0~(\text{mod}~ Y_{\gamma}),~~\text{otherwise~.} \label{5t0}\\
\end{array}
\right.
\end{eqnarray}
where
\begin{equation}\label{5t1}
T\triangleq A_{\alpha\alpha}Y+N
+\sum_{p=1}^m\frac{B_{p\alpha,\alpha}}{\lambda_p-\mu}Y_p
-\mu\xi
\end{equation}
is independent of $\alpha$ by \eqref{5b3}\eqref{5a1}. Then we assert that the subspace
\begin{equation}\label{eq-v2}
V_2\triangleq\text{Span}\{T,Y_{\gamma}|m+1\le\gamma\le n\}
\end{equation}
is parallel along $M$. According to our previous computation,
$E_i(Y_{\alpha})=0~(\text{mod}~V_2),~\forall \alpha~.$
So we need only to consider $E_i(T)$. Fix $i$ and
choose $\alpha\ne i$.
(Such $\alpha$ exists by the assumption $n-m\ge 2$,
which is the third and final time that we use it.
Recall that this condition has been used to derive
\eqref{5b2}\eqref{5b13}, i.e. $E_{\alpha}(\mu)=0=E_{\alpha}(a)$.)
Rewrite the first equality of \eqref{5t0} as
\begin{equation}\label{eq-rewrite}
T=-E_{\alpha}(Y_{\alpha})+\sum{_{\gamma}}(\cdots)Y_{\gamma}.
\end{equation}
By this clever choice of index $\alpha$
we may prove in a unified way that
\begin{eqnarray*}
E_i(T)&=&-E_i(E_{\alpha}(Y_{\alpha}))
+\sum{_{\gamma}}(\cdots)E_i(Y_{\gamma})~~(\text{mod}~Y_{\gamma})\\
&=&-E_{\alpha}(E_i(Y_{\alpha}))+[E_{\alpha},E_i](Y_{\alpha})
+\sum{_{\gamma}}(\cdots)E_i(Y_{\gamma})~~(\text{mod}~Y_{\gamma})\\
&=&-E_{\alpha}(\sum{_{\beta}} (\cdots)Y_{\beta}))+[E_{\alpha},E_i](Y_{\alpha})
+\sum{_{\gamma}}(\cdots)E_i(Y_{\gamma})~~(\text{mod}~Y_{\gamma})\\
&=&0~~(\text{mod}~V_2).
\end{eqnarray*}
This verifies our previous assertion. More precisely, we have
\begin{equation}\label{5t2}
E_p(T)=\frac{B_{p\alpha,\alpha}}{\lambda_p-\mu}T,~~
E_{\alpha}(T)=QY_{\alpha},~~\forall~p,\alpha
\end{equation}
where
\[
Q\triangleq \langle T,T\rangle=
2A_{\alpha\alpha}+\mu^2+\sum_{p=1}^m\frac{B_{p\alpha,\alpha}^2}
{(\lambda_p-\mu)^2},
\]
satisfies
\begin{equation}\label{5t3}
E_p(Q)=\frac{2B_{p\alpha,\alpha}}{\lambda_p-\mu}Q,\;
E_{\alpha}(Q)=0.
\end{equation}
One could verify \eqref{5t2} directly. But the easy way is using $\langle T,Y_{\alpha}\rangle=0$ and \eqref{5t0} to get
\begin{equation}\label{5t4}
\langle E_i(T),Y_{\alpha}\rangle
=-\langle T,E_i(Y_{\alpha})\rangle
=\left\{
\begin{array}{ll}
Q~,~\text{when}~i=\alpha;\\
0~,~\text{otherwise~.}\\
\end{array}
\right.
\end{equation}
This implies $E_p(T)\parallel T$ for any $1\le p\le m$. Then $E_p(T)$ as in \eqref{5t2} is derived by differentiating \eqref{5t1} and comparing the $\xi$ component with $T$. The formula for $E_p(Q)$ in \eqref{5t3} follows directly.
On the other hand, we know
\[
\langle E_{\alpha}(T),T\rangle=\frac{1}{2}E_{\alpha}(Q)=0,
\]
where we used \eqref{5b3} and its consequence $[E_p,E_{\alpha}]\in D_2$ together with \eqref{5b2}\eqref{5b13}\eqref{5b14}. Combined with \eqref{5t4} we have $E_{\alpha}(T)=QY_{\alpha}$.

Regarding \eqref{5t3} as a linear first-order ODE for $Q$ we see that $Q\equiv 0$ or $Q\neq 0$ on the connected
manifold $M^n$. Thus there are three possibilities for the induced metric on the fixed subspace $V_2\subset \mathbb{R}^{n+3}_1$.

\vspace{2mm}

\noindent{\bf Case 1}, $Q=0$ on $M^n$; $V_2$ is endowed with a degenerate inner product.

In this case, $\langle T,T\rangle=0$. By \eqref{5t2},
$E_p(T)\parallel T$, so $T$ determines a fixed light-like direction in $\mathbb{R}^{n+3}_1$,
which we may take to be
\[
[T]=[1,-1,0,\cdots,0]\in \mathbb{R}^{n+3}_1.
\]
This corresponds to $\infty$, the point at infinity of $\mathbb{R}^{n+1}$. Choose space-like vectors
$X_{m+1},\cdots,X_n$ so that $V_2=\text{Span}\{T,X_{m+1},\cdots,X_n\}$. We interpret the geometry of hypersurface $f:M^n\rightarrow R^{n+1}$
as below:

1)
Any $X_{\alpha}$ determines a hyperplane in $\mathbb{R}^{n+1}$
because $\langle T,X_{\alpha}\rangle=0$;

 2)
$\rm{Span}\{X_{\alpha}, (m+1\le \alpha\le n)\}$ corresponds to a
(n-m)-dimensional plane $\Sigma$ in $\mathbb{R}^{n+1}$.

 3)
$F$ is a $m$-parameter family of hyperplanes orthogonal to
the fixed plane $\Sigma$.

$f(M)$, as the envelope of this family of hyperplanes $F$, is clearly a cylinder over a hypersurface
$\widetilde{M}\subset \mathbb{R}^{m+1}$.

\vspace{2mm}
\noindent{\bf Case 2}, $Q<0$ on $M^n$; $V_2$ is a Lorentz subspace in $\mathbb{R}^{n+3}_1$.

Fix a basis $\{P_0,P_{\infty},X_{m+2},\cdots,X_n\}$
of the $(n-m+1)$-dimensional $V_2$ so that $P_0,P_{\infty}$ are light-like.
Without loss of generality we may assume
\[
P_0=(1,1,0,\cdots,0),~~P_{\infty}=(1,-1,0,\cdots,0).
\]
Using the stereographic projection $\sigma$ they correspond to
the origin $O$ and the point at infinity $\infty$ of the flat
$\mathbb{R}^{n+1}$, respectively.
We interpret $F$ and $V_2$ in terms of the geometry of $\mathbb{R}^{n+1}$:

1)
$\rm{Span}\{X_{\alpha}:$ $m+2\le \alpha\le n\}$ corresponds to a coordinate plane
$\mathbb{R}^{n-m-1}\subset\mathbb{R}^{n+1}$, because $X_{\alpha}$ must be space-like and orthogonal to $P_0,P_{\infty}$.

2)
$F$ is a $m$-parameter family of hyperplanes (passing $O$ and $\infty$) and orthogonal to this fixed $\mathbb{R}^{n-m-1}$.

Based on the fact 1), $f(M)$, the envelope of $F$,
is a cylinder over a $(m+1)$-dimensional hypersurface in $\mathbb{R}^{m+2}$ (the orthogonal complement of the previous $\mathbb{R}^{n-m-1}$); moreover, the fact 2) means that $f(M)$ is a cone (with vertex $O$) over a $m$-dimensional hypersurface in $S^{m+1}$.

\vspace{2mm}
\noindent{\bf Case 3}, $Q>0$ on $M^n$; $V_2$ is a space-like subspace.

Without loss of generality we assume that
$P_{\infty}=(1,-1,0,\cdots,0)$ is contained in the orthogonal complement of $V_2$. As before we make the following interpretation:

1)
$V_2$ corresponds to a $m$-dimensional plane $\mathbb{R}^m\subset\mathbb{R}^{n+1}$.

2)
$F$ is a $(n-m)$-parameter family of hyper-spheres orthogonal to
this fixed plane $\mathbb{R}^m$ with centers locating on it. Thus $F$ envelops a rotational hypersurface $f(M)$ (over a hypersurface in half-space $\mathbb{R}_+^{m+1}$).

Sum together we complete the proof to the Reduction Theorem.
\end{proof}

\begin{Remark}
It is noteworthy that we may introduce
\[P\triangleq QY-T\]
which satisfies $\langle P,T\rangle=0,\langle P,Y_{\alpha}\rangle=0,
\langle P,P\rangle=-Q$. So $P\bot V_2$ and $QY=T+P$
is an orthogonal decomposition.
Hence a direct proof for case 2 and 3 is to define
\[
\bar{P}=\frac{P}{\sqrt{|Q|}}~,
~~\theta=\frac{T}{\sqrt{|Q|}}~,
~~\langle \bar{P},\bar{P}\rangle=-\langle\theta,\theta\rangle=\pm 1~.
\]
Either of them gives a map into the sphere or
the hyperbolic space. Then $M^n=L^m\times N^{n-m}$
is mapped to the lightcone of $\mathbb{R}^{n+3}_1$ by
\[
Y=\frac{-1}{\sqrt{|Q|}}(\bar{P},\theta)~\in
\mathbb{R}^{n+3}_1=V_2^{\bot}\oplus V_2
\]
as a warped product of these two maps
($Q$ depends only on the component of $\widetilde{M}^m$).
Clearly such hypersurfaces are cones or rotational hypersurfaces.
\end{Remark}
The construction of cylinders, cones and rotational
hypersurfaces exists for any index $1\le m \le n-1$.
From this viewpoint the condition $(2)$ that $2\le n-m\le n-2$
in our Reduction Theorem~\ref{retheorem} is unsatisfying,
not only conceptually, but also in that it limits the
possible application.

Upon closer examination we find that when $m=n-1$
(the M\"obius principal curvature $\mu$ is simple)
one could not find a satisfying version of the
Reduction Theorem. In particular it seems unavoidable to assume
that $\lambda_1,\cdots,\lambda_{n-1}$ (and $\mu$) be distinct
(which seems to be a quite unnatural condition),
so that we can derive
\[
\omega_{pq}=\sum_{r=1}^{n-1}
\frac{B_{pq,r}}{\lambda_p-\lambda_q}\omega_r
\]
(similar to \eqref{5b4}) and use it to compute $E_i(T)$.
(As pointed out before \eqref{eq-rewrite}
in our previous proof of Theorem~\ref{retheorem},
the condition $m\le n-2$ has been used several times,
in particular to show
$E_i(T)=0(\text{mod}~V_2)$ before \eqref{5t2}.)
It seems preferable to verify whether the subspace
$V_1$ or $V_2$ defined in \eqref{eq-v1}\eqref{eq-v2}
is invariant or not
when the Reduction Theorem could not apply directly.

On the other hand, our Reduction Theorem can be generalized to
the case $m=1$ with some modification on the assumptions.
\begin{Theorem}\label{retheorem2}
Let $f: M^n\rightarrow R^{n+1}~~(n\geq 3)$ be a hypersurface in
 $(n+1)-$dimensional Euclidean space with a principal curvature
of multiplicity $n-1$. Below are equivalent:

$(1)$ $f$ is M\"{o}bius congruent to
a cylinder, or a cone, or a rotation hypersurface over
a curve $\gamma\subset N^2(\epsilon)$.

$(2)$ The M\"{o}bius form $\Phi=\sum_i C_i\omega_i$ of $f$ is closed.
\end{Theorem}
\begin{proof}
Write out $\Phi=\sum_iC_i\omega_i$, the coefficient matrices $(B_{ij})$ of the M\"{o}bius second fundamental form and $(A_{ij})$ of the Blaschke tensor under any orthonormal basis
$\{E_1,\cdots,E_n\}$ with respect to the M\"{o}bius metric $g$
and dual basis $\{\omega_1,\cdots,\omega_n\}$. Notice
\begin{equation*}
d\Phi=\sum_idC_i\wedge\omega_i+\sum_iC_id\omega_i
=\sum_{ij}C_{i,j}\omega_j\wedge\omega_i
\end{equation*}
and the integrability equation \eqref{equa2}. Then the following are obviously equivalent:

1) $\Phi$ is a closed 1-form;

2) $C_{i,j}$ define a symmetric tensor;

3) matrices $(B_{ij})$ and $(A_{ij})$ commute;

4)$(B_{ij})$ and $(A_{ij})$ can be diagonalized simultaneously.

Suppose $f$ has a principal curvature
of multiplicity $n-1$ and $\Phi$ is closed.
Then we can choose $\{E_1,\cdots,E_n\}$ such that
\[
(B_{ij})=\text{diag}(\lambda,\mu,\cdots,\mu),~~ (A_{ij})=\text{diag}(a_1,a_2,\cdots,a_n).
\]
We are almost in the same context as in the proof of Theorem~\ref{retheorem} with $m=1$ and here we still assume
$1\leq i,j,k \leq n; 2\leq \alpha,\beta,\gamma \leq n.$
In particular \eqref{5b1} still holds true and we have
$B_{\alpha\beta,\alpha}=0$ for any $\alpha\ne \beta$.

Using \eqref{equa6} we know
$\lambda=\frac{n-1}{n},\mu=\frac{-1}{n}$ identically. Differentiate them. We get
\begin{equation*}
\begin{split}
B_{11,\alpha}&=0,~~\forall\alpha,\\
0=E_{\beta}(\mu)&=B_{\alpha\alpha,\beta}
=B_{\alpha\beta,\alpha}
+\delta_{\alpha\alpha}C_{\beta}-\delta_{\alpha\beta}C_{\alpha}
=C_{\beta},~~\forall \alpha\ne\beta.
\end{split}
\end{equation*}
This looks like \eqref{5b2} and we also
use \eqref{5b1}\eqref{equa3}. But the assumption is different.
Anyway we find that the condition $(3)$ in
the Reduction Theorem~\ref{retheorem} is satisfied.
Although here $m=1$ violates the condition $(2)$, we observe that
$m\ge 2$ is only used only once in that proof to derive
\eqref{5b12}:
\[A_{q\alpha}=0,\]
which is an established fact at here already.
Thus the previous proof to Theorem~\ref{retheorem}
after \eqref{5b12} is still valid. The same argument shows
that $f$ is reducible.

Conversely, if $f$ could be reduced to
Example (\ref{rre12}), (\ref{rre13}), or (\ref{rre14}),
by the computations in the previous section we know that
 $(B_{ij})$ and $(A_{ij})$ can be diagonalized simultaneously,
 thus $C$ is closed.
 This finishes the proof to Theorem~\ref{retheorem2}.
\end{proof}
\begin{Remark}
In \cite{guo1}, Guo and Lin obtained a classification
of hypersurfaces with two distinct principal curvatures
and closed M\"obius form $\Phi$, which included our
Theorem~\ref{retheorem2}.
We give an alternative proof here not only to be self-contained,
but also because this proof looks simpler and unified with
the Reduction Theorem~\ref{retheorem}.
\end{Remark}

\section{Algebraic characteristics of second fundamental forms of deformable hypersurface pairs}
Let $f,\bar{f}: M^n\rightarrow R^{n+1}~~(n\geq 4)$ be two
hypersurfaces without
umbilics. If they induce the same M\"{o}bius metric, i.e., $g=\bar{g}$,
then the M\"{o}bius second fundamental forms $B$ of $f$, and
$\bar{B}$ of $\bar{f}$, have specific algebraic characteristics.
The algebraic result is as below:
\begin{Theorem}\label{th1}
Let $V$ be a $n$-dimensional vector space $(n\geq 4)$, and
$B,\bar{B}: V\times V\rightarrow R$ be two bilinear symmetric
functions. Let $\{e_1,\cdots,e_n\}$ be an orthonormal basis of $V$,
and write $B(e_i,e_j)=B_{ij}, \bar{B}(e_i,e_j)=\bar{B}_{ij}$. Denote
\begin{equation}\label{s0}
\begin{split}
S_{ijkl}=&B_{ik}B_{jl}-B_{il}B_{jk}\\
&+\frac{1}{n-2}\sum_m
\{\delta_{ik}B_{jm}B_{ml}+\delta_{jl}B_{im}B_{mk}
-\delta_{il}B_{jm}B_{mk}-\delta_{jk}B_{im}B_{ml}\}.
\end{split}
\end{equation}
Obviously this defines a tensor $S: V^4\rightarrow R$ associated
with $B$. $\bar{S}$ and $\bar{S}_{ijkl}$ are defined similarly for
$\bar{B}$. Assume $S=\bar{S}$, i.e.
\begin{equation*}
S_{ijkl}=\bar{S}_{ijkl},\;\;\forall\, 1\le i,j,k,l\le n.
\end{equation*}
Then either $B$ and $\bar{B}$ can be diagonalized simultaneously, or
there exists an orthonormal basis $\{e_1,\cdots,e_n\}$ of $V$ such
that
\begin{eqnarray*}
\{\bar{B}_{ij}\} =\text{diag}(\bar{\lambda}_1,\bar{\lambda}_2,\bar{\mu},
\cdots,\bar{\mu}), \{B_{ij}\}=\left(
    \begin{array}{ccccc}
      B_{11} & B_{12}& 0 & \cdots &0\\
      B_{21} &B_{22} & 0 & \cdots & 0 \\
      0 & 0 & \mu& \cdots& 0 \\
      \vdots& \vdots & \vdots & \ddots & \vdots \\
      0 & 0 & 0 & \cdots & \mu\\
    \end{array}
  \right),
\end{eqnarray*}
where $\bar{\lambda}_1\ne \bar{\lambda}_2, \mu=\pm\bar\mu$. In the
last case there exist an eigenvalue of $\bar{B}$ with multiplicity at
least $n-2$.
\end{Theorem}
To prove Theorem \ref{th1}, we need the following two lemmas.
\begin{lemma}\label{lem1}
Given $n\ge 4$. Assumptions as in Theorem~\ref{th1} except that
$dim(V)=l, 3\le l\le n$. That means we still have the fraction
$\frac{1}{n-2}$ in the expression \eqref{s0}, yet the range of those
indices is from $1$ to $l$. Then we can find an orthonormal basis of
$V$ so that
$\{\bar{B}_{ij}\}=\text{diag}(\bar{\lambda}_1,\cdots,\bar{\lambda}_l)$ and
$B_{ij}=0$ for some $i\ne j$ (i.e. there is at least one
off-diagonal element of $\{B_{ij}\}$ equals to zero).
\end{lemma}
\begin{proof} Since $\bar{B}$ is symmetric, we can always diagonalize
it as
$\{\bar{B}_{ij}\}=\text{diag}(\bar{\lambda}_1,\cdots,\bar{\lambda}_l)$ with
respect to an orthonormal basis of $V$. If there has been some
$B_{ij}=0$ with $i\ne j$ at the same time, we are done. Otherwise,
suppose all the off-diagonal elements of $\{B_{ij}\}$ are non-zero.
In this case we make the following

{\bf Assertion:} $~\{\bar{\lambda}_1,\cdots,\bar{\lambda}_l\}$
could not be all distinct. \\
Hence there must exist two equal eigenvalues
$\bar{\lambda}_\alpha=\bar{\lambda}_\beta$, which enables us to
rotate the basis vectors $\{e_\alpha,e_\beta\}$ properly in the
plane $span\{e_\alpha,e_\beta\}$ and to obtain a new orthonormal
basis of $V$, so that $\{\bar{B}_{ij}\}$ is still a diagonal matrix
and $B_{\alpha\beta}=0$. This completes the proof.

To prove the assertion above (on condition that $B_{ij}\ne 0,
\forall~i\ne j$), we substitute the expressions of $S,\bar{S}$ and
$\{\bar{B}_{ij}\}=\text{diag}(\bar{\lambda}_1,\cdots,\bar{\lambda}_l)$ into
the equality
\[
S_{\alpha 1\alpha 1}-S_{\alpha 2\alpha 2} =\bar{S}_{\alpha 1\alpha
1}-\bar{S}_{\alpha 2\alpha 2}, ~~~\forall~3\le\alpha\le l.
\]
As the result we obtain
\begin{align}
B_{\alpha\alpha}&(B_{11}-B_{22})-(B_{1\alpha}^2-B_{2\alpha}^2)
+\frac{1}{n-2}\sum_{m=1}^l(B_{1m}^2-B_{2m}^2) \label{s2}\\
&=(\bar\lambda_1-\bar\lambda_2)[\bar\lambda_{\alpha}
+\frac{1}{n-2}(\bar\lambda_1+\bar\lambda_2)],
~~~~~\forall~3\le\alpha\le l.  \notag
\end{align}
In the following let the range of the index $\alpha$ be
$3\le\alpha\le l$. We want to show that the left hand side of
\eqref{s2} vanishes. First note that
$\{\bar{B}_{ij}\}=\text{diag}(\bar{\lambda}_1,\cdots,\bar{\lambda}_l)$
implies
\[\bar{S}_{ijik}=0,\]
when $i,j,k$ are distinct. It follows from the equality
$S_{ijik}=\bar{S}_{ijik}$ that
\begin{equation}\label{s3}
B_{ii}B_{jk}-B_{ij}B_{ik}+\frac{1}{n-2} \sum_{m=1}^l
B_{jm}B_{km}=0,~~~~\forall~\text{distinct}~i,j,k.
\end{equation}
Hence
\begin{align}
\frac{B_{11}}{B_{12}}-\frac{B_{1\alpha}}{B_{2\alpha}}
&=-\frac{1}{n-2}\cdot\frac{1}{B_{12}B_{2\alpha}}\cdot
\sum_{m=1}^l B_{2m}B_{m\alpha} \label{s4}\\
&=-\frac{1}{n-2}\left[\frac{B_{1\alpha}}{B_{2\alpha}}
+\frac{B_{22}}{B_{12}}+\sum_{m=3}^l
\frac{B_{m\alpha}}{B_{12}}\cdot\frac{B_{2m}}{B_{2\alpha}}
\right].\notag
\end{align}
Similarly one can find
\begin{align}
\frac{B_{22}}{B_{12}}-\frac{B_{2\alpha}}{B_{1\alpha}}
&=-\frac{1}{n-2}\cdot\frac{1}{B_{12}B_{1\alpha}}\cdot
\sum_{m=1}^l B_{1m}B_{m\alpha} \label{s5}\\
&=-\frac{1}{n-2}\left[\frac{B_{2\alpha}}{B_{1\alpha}}
+\frac{B_{11}}{B_{12}}+\sum_{m=3}^l
\frac{B_{m\alpha}}{B_{12}}\cdot\frac{B_{1m}}{B_{1\alpha}}
\right].\notag
\end{align}
Taking $\eqref{s4}-\eqref{s5}$ yields
\[
\left[ \frac{B_{11}-B_{22}}{B_{12}}-
\frac{B_{1\alpha}}{B_{2\alpha}}+\frac{B_{2\alpha}}{B_{1\alpha}}
\right]\left(1-\frac{1}{n-2}\right) =-\frac{1}{n-2}\sum_{m=3}^l
\frac{B_{m\alpha}}{B_{12}} \left(\frac{B_{2m}}{B_{2\alpha}}
-\frac{B_{1m}}{B_{1\alpha}}\right) =0,
\]
due to $B_{2m}B_{1\alpha}-B_{2\alpha}B_{1m}
=S_{21m\alpha}=\bar{S}_{21m\alpha}=0$ when $\bar{B}$ is diagonal and
$m,\alpha\ge3$. We conclude
\[
\frac{B_{11}-B_{22}}{B_{12}}=
\frac{B_{1\alpha}}{B_{2\alpha}}-\frac{B_{2\alpha}}{B_{1\alpha}} =b,
\]
for some constant $b$. It follows that
\begin{align*}
&B_{\alpha\alpha}(B_{11}-B_{22})-(B_{1\alpha}^2-B_{2\alpha}^2)
+\frac{1}{n-2}\sum_{m=1}^l(B_{1m}^2-B_{2m}^2)\\
=~&B_{\alpha\alpha}(B_{11}-B_{22})-(B_{1\alpha}^2-B_{2\alpha}^2)
+\frac{1}{n-2}(B_{11}^2-B_{22}^2)
+\frac{1}{n-2}\sum_{m=3}^l(B_{1m}^2-B_{2m}^2)\\
=~&B_{\alpha\alpha}\cdot bB_{12}-b\cdot B_{1\alpha}B_{2\alpha}
+\frac{1}{n-2}(B_{11}+B_{22})\cdot bB_{12}
+\frac{1}{n-2}\sum_{m=3}^l(b\cdot B_{1m}B_{2m})\\
=~&b\left[B_{\alpha\alpha}B_{12}-B_{1\alpha}B_{2\alpha}
+\frac{1}{n-2}\sum_{m=1}^l B_{1m}B_{2m}\right]=0,
\end{align*}
by \eqref{s3}. From \eqref{s2} we have
\begin{equation}\label{lambda}
(\bar\lambda_1-\bar\lambda_2)[\bar\lambda_{\alpha}
+\frac{1}{n-2}(\bar\lambda_1+\bar\lambda_2)]=0.
\end{equation}
So either $\bar\lambda_1=\bar\lambda_2$, or $\bar\lambda_{\alpha}
=-\frac{1}{n-2}(\bar\lambda_1+\bar\lambda_2)],
~\forall~3\le\alpha\le l.$ This verifies the assertion when $l\ge
4$.

The only case unsolved is when $l=3$. This time \eqref{lambda} takes
the form
\[
(\bar\lambda_1-\bar\lambda_2)[\bar\lambda_3
+\frac{1}{n-2}(\bar\lambda_1+\bar\lambda_2) ]=0.
\]
Taking permutation of the indices $1,2,3$ yields two other similar
formulas. Now it is easy to prove that
$\bar\lambda_1,\bar\lambda_2,\bar\lambda_3$ can not be all distinct
by contradiction. Hence the proof to Lemma \ref{lem1} is finished.
\end{proof}
\begin{Remark}\label{rem-dim3}
Note that in the proof above we used the fact
$\frac{1}{n-2}\ne 1$ at two places. Thus the condition $n\ge 4$
is necessary. On the other hand, by the integrability equations (\ref{equa4})
the Weyl conformal tenor associated with the M\"obius metric $g$
can be expressed by the M\"{o}bius invariants as below:
\begin{equation*}
\begin{split}
C_{ijkl}&=B_{ik}B_{jl}-B_{il}B_{jk}-\frac{1}{n(n-2)}(\delta_{ik}\delta_{jl}-\delta_{jk}\delta_{il})\\
&+\frac{1}{n-2}\sum_m
\{\delta_{ik}B_{jm}B_{ml}+\delta_{jl}B_{im}B_{mk}
-\delta_{il}B_{jm}B_{mk}-\delta_{jk}B_{im}B_{ml}\}\\
&=S_{ijkl}-\frac{1}{n(n-2)}(\delta_{ik}\delta_{jl}
-\delta_{jk}\delta_{il}).
\end{split}
\end{equation*}
It is well known that the Weyl conformal tenor vanishes on three dimensional Riemannian manifold. Therefore when $n=3$, $S_{ijkl}=\frac{1}{3}(\delta_{ik}\delta_{jl}
-\delta_{jk}\delta_{il})=\bar{S}_{ijkl}$ is a trivial identity.
\end{Remark}
\begin{lemma}\label{lem2}
Assumptions as in Lemma~\ref{lem1}. By the conclusion above, without
loss of generality we may suppose that for a given orthonormal basis
of $V$ there are
$\{\bar{B}_{ij}\}=\text{diag}(\bar{\lambda}_1,\cdots,\bar{\lambda}_l)$ and
$B_{ij}=0$ for some $i\ne j$. Then there exists a properly chosen
new orthonormal basis of $V$, with respect to which
$\{\bar{B}_{ij}\}$ is still diagonal and
\begin{eqnarray*}
\{B_{ij}\}=\left(
    \begin{array}{cccc}
      B_{11} & \cdots & B_{1,l-1} & 0\\
      \vdots & \ddots & \vdots & \vdots \\
      B_{l-1,1} & \cdots & B_{l-1,l-1} & 0 \\
      0 & \cdots & 0 & B_{ll}\\
    \end{array}
  \right)
\end{eqnarray*}
is a semi-diagonal matrix.
\end{lemma}
\begin{proof}
For simplicity denote $k=l-1$. Without loss of generality we may
assume that the off-diagonal element $B_{kl}=0$.

First we consider the easy case $l=3$. As in \eqref{s3}, we have
\[
0=B_{11}B_{23}-B_{12}B_{13}+\frac{1}{n-2} \sum_{m=1}^3 B_{2m}B_{m3}
=\left(\frac{1}{n-2}-1\right) B_{12}B_{13},
\]
because $B_{23}=0$ as assumed. It follows that either $B_{12}=0$ or
$B_{13}=0$, and the conclusion is proved.

In general, when $l\ge 4$, for any $i<j<k=l-1$ there is
\begin{equation}\label{s6}
0=\bar{S}_{ikjl}=S_{ikjl}=B_{ij}B_{kl}-B_{il}B_{kj} =-B_{il}B_{kj}.
\end{equation}
If $B_{il}=0$ for any $i< k=l-1$, then all the off-diagonal elements
in the $l$-th column and the $l$-th row vanish, and we are done.
Otherwise, suppose $B_{1l}\ne 0$ without loss of generality. Then by
\eqref{s6}, $B_{jk}=0,~\forall~1<j<k=l-1$. Using this result and
$B_{lk}=0, B_{1l}\ne 0$, we may prove $B_{1k}=0$ by \eqref{s3}:
\[
0=B_{11}B_{kl}-B_{1k}B_{1l}+\frac{1}{n-2} \sum_{j=1}^l B_{jk}B_{jl}
=\left(\frac{1}{n-2}-1\right) B_{1k}B_{1l}.
\]
So $B_{jk}=0,~\forall~j\ne k$. That means all the off-diagonal
elements in the $k$-th column and the $k$-th row vanish.
Interchanging the basis vectors $e_k$ and $e_l$ gives the desired
result. The proof to Lemma \ref{lem2} is finished.
\end{proof}

\begin{proof}[Proof to Theorem~\ref{th1}]
From Lemma~\ref{lem1} and Lemma~\ref{lem2} and by induction it is
easy to see that $\{B_{ij}\},\{\bar{B}_{ij}\}$ can be diagonalized
simultaneously except that $B_{12}$ might be non-zero.

Denote
$\{\bar{B}_{ij}\}=\text{diag}(\bar{\lambda}_1,\cdots,\bar{\lambda}_n)$ as
before. When $B_{12}=0$ it is the first case in the conclusion. If
$B_{12}\ne 0$ yet $\bar{\lambda}_1=\bar{\lambda}_2$, one might
rotate the basis vectors $\{e_1,e_2\}$ properly in the plane
$span\{e_1,e_2\}$ and obtain a new orthonormal basis of $V$ such
that $\{\bar{B}_{ij}\}$ is invariant and $B_{12}=0$, hence we are
also done. The final part of the proof is to show that when
$B_{12}\ne 0$ and $\bar{\lambda}_1\ne \bar{\lambda}_2$, $\{B_{ij}\}$
and $\{\bar{B}_{ij}\}$ must have the desired multiplicities of their
eigenvalues.

Again by \eqref{s3}, $\forall~3\le\alpha\le n$,
\[
0=B_{\alpha\alpha}B_{12}-B_{\alpha 1}B_{\alpha 2}
+\frac{1}{n-2}\sum_{m=1}^n B_{1m}B_{m2} =B_{12}\left[
B_{\alpha\alpha}+\frac{1}{n-2}(B_{11}+B_{22})\right ].
\]
Thus $B_{\alpha\alpha}=-\frac{1}{n-2}(B_{11}+B_{22})=\mu$ for any $\alpha\ge 3$. So $\{B_{ij}\}$ has the desired
form. As a by-product we find that
\begin{equation}\label{tr}
tr(B)=B_{11}+B_{22}+(n-2)\mu=0.
\end{equation}
Taking use of the fact above and the equalities $S_{1\alpha
1\alpha}=\bar{S}_{1\alpha 1\alpha}, S_{2\alpha
2\alpha}=\bar{S}_{2\alpha 2\alpha}$, we have
\begin{align}
B_{11}\mu+\frac{1}{n-2}\left ( \lambda^2+B_{11}^2+B_{12}^2\right )
&= \bar\lambda_1\bar\lambda_{\alpha}+\frac{1}{n-2}\left (
\bar\lambda_{\alpha}^2+\bar\lambda_1^2\right ), \label{b1}\\
B_{22}\mu+\frac{1}{n-2}\left ( \lambda^2+B_{22}^2+B_{12}^2\right )
&= \bar\lambda_2\bar\lambda_{\alpha}+\frac{1}{n-2}\left (
\bar\lambda_{\alpha}^2+\bar\lambda_2^2\right ), \label{b2}
\end{align}
for any $\alpha \ge 3$. Taking $\eqref{b1}-\eqref{b2}$ yields
\[
(B_{11}-B_{22})\left [
B_{\alpha\alpha}+\frac{1}{n-2}(B_{11}+B_{22})\right ]
=(\bar\lambda_1-\bar\lambda_2)\left [
\bar\lambda_{\alpha}+\frac{1}{n-2}
(\bar\lambda_1+\bar\lambda_2)\right ].
\]
The left hand side vanishes by \eqref{tr}. It follows that
$\bar\lambda_{\alpha}=-\frac{1}{n-2}
(\bar\lambda_1+\bar\lambda_2)=\bar\lambda$ for all
$\alpha\ge 3$ (keep in mind that $\bar{\lambda}_1\ne
\bar{\lambda}_2$ at here) and $tr(\bar{B})=0$. Finally
$S_{3434}=\bar{S}_{3434}$ implies $\mu^2=\bar\mu^2$.
This finishes the proof to Theorem \ref{th1}.
\end{proof}

\section{Hypersurfaces with low multiplicities: rigidity} Let $f,\bar{f}: M^n\rightarrow R^{n+1}(n\geq
4)$ be two hypersurfaces without umbilics. In this section and the
following two, the M\"{o}bius
invariants of $f$ will be denoted by $\{A,B,C\}$ and those of $\bar{f}$ by $\{\bar{A},\bar{B},\bar{C}\}$.
\begin{Theorem}\label{thm61}
Let $f,\bar{f}: M^n\rightarrow R^{n+1}(n\geq 4)$ be two immersed
hypersurfaces without umbilics, whose principal curvatures have constant multiplicities. Assume that they induce the same
M\"{o}bius metrics $g$,
and all principal curvatures of $B$ have multiplicity less than
$n-2$ everywhere. Then $f$ is M\"{o}bius
congruent to $\bar{f}$.
\end{Theorem}
We divide our proof into two parts. The case of dimension $n=4$ is different from higher dimensional case ($n\ge 5$) and need to be discussed separately. Before that
we make some preparation first.

The same M\"{o}bius metric $g$ for $f,\bar{f}$ determines
the same curvature tensor $R_{ijkl}$.
By the integrability equations \eqref{equa4}\eqref{equa5},
the conclusion of Theorem~\ref{th1} applies to
the M\"{o}bius second fundamental forms $B,\bar{B}$.
Since the multiplicities of all principal
curvatures are less than $n-2$ at here by assumption,
Theorem~\ref{th1} guarantees that locally we can
choose an orthonormal basis
$\{E_1,\cdots,E_n\}$ with respect to $g$ such that
\begin{eqnarray*}
\{B_{ij}\}=\text{diag}(\lambda_1,\cdots,\lambda_n);
\{\bar{B}_{ij}\}=\text{diag}(\bar{\lambda}_1,\cdots,\bar{\lambda}_n).
\end{eqnarray*}
Now \eqref{equa4}\eqref{equa5} imply
\begin{equation}\label{t}
\lambda_i\lambda_j+\frac{1}{n-2}(\lambda_i^2+\lambda_j^2)
=\bar{\lambda}_i\bar{\lambda}_j+\frac{1}{n-2}(\bar{\lambda}_i^2
+\bar{\lambda}_j^2),~~\forall~i\ne j.
\end{equation}
Changing the subscript of \eqref{t} and taking difference, we get
\begin{equation}\label{t1}
(\lambda_i-\lambda_k)[\lambda_j+\frac{1}{n-2}(\lambda_i+\lambda_k)]
=(\bar{\lambda}_i-\bar{\lambda}_k)[\bar{\lambda}_j+
\frac{1}{n-2}(\bar{\lambda}_i+\bar{\lambda}_k)],
~~\forall~\text{distinct}~i,j,k.
\end{equation}
To obtain the rigidity result we need only to show that
$\bar{B}=\pm B$; reverse the direction of the
normal vector field of $f$ if necessary we will have
$\bar{B}= B$, which shows that $f$ is congruent to $\bar{f}$
by the fundamental theorem~\ref{fundthe}.
\begin{PROPOSITION}\label{prop62}
The conclusion of Theorem~\ref{thm61} is valid when the dimension $n\ge 5$.
\end{PROPOSITION}
\begin{proof}
We assert that there is a linear relation between
$\lambda_j$ and $\bar{\lambda}_j$, i.e. there exists constants
$b,c$ such that
\[\bar{\lambda}_j=b\lambda_j+c,~~\forall~1\le j\le n.\]
In other words, regard $p_j=(\lambda_j,\bar{\lambda}_j)$ as
coordinates of $n$ points on a plane, then these $n$ points
are collinear.

Without loss of generality assume that $\lambda_1\ne \lambda_2$.
We just show $p_3=(\lambda_3,\bar{\lambda}_3)$ is collinear with
$p_1=(\lambda_1,\bar{\lambda}_1),p_2=(\lambda_2,\bar{\lambda}_2)$.
(For any other index $j\ne 1,2,3$ the proof is the same.)
Now we need to consider two cases separately.

In the first case, $n\geq 5$ and the highest multiplicity of principal curvatures is less than $n-3$. We can find $\lambda_i\ne\lambda_k$
which are distinct from $\{1,2,3\}$. Fix $i,k$ in \eqref{t1},
we see that all other $(\lambda_j,\bar{\lambda}_j)$ ($j\ne i,k$)
satisfies a non-trivial linear equation \eqref{t1}.
In particular, $p_1,p_2,p_3$ are collinear.

In the second case, $\lambda_i$ might be a constant
for any indices $i\ne 1,2,3$.
(Note that $\lambda_i\ne\lambda_1,\lambda_2,\lambda_3$.
Otherwise there will be
a principal curvature of multiplicity at least $n-2$,
contradiction. Yet $\lambda_3$ might be equal to either of
$\lambda_1,\lambda_2$.) Fix $i=1,k=5$, we have
$\lambda_1\ne \lambda_5$ and by \eqref{t1}
we know
\[p_2=(\lambda_2,\bar{\lambda}_2),p_3=(\lambda_3,\bar{\lambda}_3),
p_4(\lambda_4,\bar{\lambda}_4)~ \text{are collinear}.\]
Similarly we know $\{p_1,p_2,p_4\}$ and $\{p_1,p_3,p_4\}$
are collinear triples. This guarantees that $\{p_1,p_2,p_3\}$
(and other $p_j$'s) are collinear and finishes the proof to
our assertion.

Now we know $\bar{\lambda}_j=b\lambda_j+c$ for constants $b,c$ and for any $j$. The fact $\sum_j \lambda_j=0=\sum_j \bar\lambda_j$
(the first identity in \eqref{equa6}) implies $c=0$.
Using the second identity $\sum_j \lambda_j^2=\frac{n-1}{n}=\sum_j \bar\lambda_j^2$ in \eqref{equa6} we conclude that $b=\pm 1$.
This completes the proof to Proposition~\ref{prop62}.
\end{proof}
\begin{PROPOSITION}\label{prop63}
The conclusion of Theorem~\ref{thm61} is valid when dimension $n=4$.
\end{PROPOSITION}
\begin{proof}
First note that when $n=4$ and the highest multiplicity is
less than $n-2=2$, four principal curvatures of $f$ are distinct. Consider
\begin{equation*}
\{B_{ij}\}=\text{diag}(\lambda_1,\lambda_2,\lambda_3,\lambda_4);~
\{\bar{B}_{ij}\}=\text{diag}(\bar{\lambda}_1,\bar{\lambda}_2,
\bar{\lambda}_3,\bar{\lambda}_4).
\end{equation*}
By \eqref{t} and $n=4$ we have
\begin{equation}\label{pm}
\lambda_i+\lambda_j=\pm(\bar\lambda_i+\bar\lambda_j),~
i\ne j.
\end{equation} We assert that there are four possibilities on
each and every point of $M$:
\begin{equation*}
\begin{split}
&(1)\;\; B=\pm \bar{B};\\
&(2)\;\;
\{B_{ij}\}=\text{diag}(\lambda_1,\lambda_2,\lambda_3,\lambda_4),\{\bar{B}_{ij}\}=\text{diag}(\pm\lambda_2,\pm\lambda_1,\pm\lambda_4,\pm\lambda_3);\\
&(3)\;\;
\{B_{ij}\}=\text{diag}(\lambda_1,\lambda_2,\lambda_3,\lambda_4),\{\bar{B}_{ij}\}=\text{diag}(\pm\lambda_3,\pm\lambda_4,\pm\lambda_1,\pm\lambda_2);\\
&(4)\;\;
\{B_{ij}\}=\text{diag}(\lambda_1,\lambda_2,\lambda_3,\lambda_4),\{\bar{B}_{ij}\}=\text{diag}(\pm\lambda_4,\pm\lambda_3,\pm\lambda_2,\pm\lambda_1).
\end{split}
\end{equation*}
Suppose $B\ne\pm \bar{B}$. Consider a special case
\[ \lambda_1+\lambda_2=(\bar\lambda_1+\bar\lambda_2),~
\lambda_1+\lambda_3=-(\bar\lambda_1+\bar\lambda_3),~
\lambda_1+\lambda_4=-(\bar\lambda_2+\bar\lambda_3).\]
Taking sum of these three equalities and using the fact
$\sum_j \lambda_j=0=\sum_j \bar\lambda_j$ we get
$\lambda_1=\bar\lambda_2$. Substitute this back and use
$\sum_j \lambda_j=0=\sum_j \bar\lambda_j$ again. We
conclude that this is case $(2)$. Reversing either of the
normal vector fields and taking permutations reduce other
possibilities to this special case. This verifies our assertion.

We need to exclude possibility $(2)$ by contradiction.
Other cases are similar.
This time $\lambda_1+\lambda_2=0$ automatically implies
$B=\pm\bar{B}$ by $\sum_j \lambda_j=0=\sum_j \bar\lambda_j$.
So we need only to find contradiction when $\lambda_1\ne\pm\lambda_2, \lambda_3\ne\pm\lambda_4$ and
\[\{B_{ij}\}=\text{diag}(\lambda_1,\lambda_2,\lambda_3,\lambda_4),
\{\bar{B}_{ij}\}
=\text{diag}(\pm\lambda_2,\pm\lambda_1,\pm\lambda_4,\pm\lambda_3),\]
under a locally
orthonormal basis $\{E_1,\cdots,E_4\}$ with respect to $g$.

Using the covariant derivative of $B$ and $\bar{B}$, we get
\begin{equation}\label{40}
(\lambda_i-\lambda_j)\omega_{ij}=\sum_kB_{ij,k}\omega_k,~~
(\bar{\lambda}_i-\bar{\lambda}_j)\omega_{ij}
=\sum_k\bar{B}_{ij,k}\omega_k,  ~~i\ne j.
\end{equation}
So
\begin{equation}\label{41}
(\bar{\lambda}_i-\bar{\lambda}_j)B_{ij,k}
=(\lambda_i-\lambda_j)\bar{B}_{ij,k},~~~i\ne j.
\end{equation}
Consequently, there is
\begin{equation}\label{42}
B_{12,k}=-\bar{B}_{12,k},~~B_{34,k}=-\bar{B}_{34,k},
\end{equation}
because $\bar{\lambda}_1-\bar{\lambda}_2
=\lambda_2-\lambda_1\ne 0,\bar{\lambda}_3-\bar{\lambda}_4
=\lambda_4-\lambda_3\ne 0$.
It follows that
\begin{equation}\label{43}
B_{ij,k}=\bar{B}_{ij,k}=0,~~\text{when $i,j,k$ are distinct}.
\end{equation}
To verify \eqref{43}, consider the case when $i=1,j=2,k=3$.
If $B_{12,3}\ne 0$, then $B_{13,2}=-\bar{B}_{13,2}\ne 0$.
(Note $B_{ij,k}=B_{ik,j}$ when $i,j,k$ are distinct
by \eqref{equa3}.) Combined with \eqref{41}, there should be
$\lambda_1-\lambda_3=\bar{\lambda}_3-\bar{\lambda}_1
=\lambda_4-\lambda_2$. Yet this implies $\lambda_1+\lambda_2=0$
under our condition $\lambda_1+\lambda_2+\lambda_3+\lambda_4=0$,
which contradicts the assumption
$\lambda_1\ne\pm\lambda_2$. Other cases are verified similarly.
As a corollary of \eqref{40} and \eqref{43},
\begin{equation}\label{44}
(\lambda_i-\lambda_j)\omega_{ij}
=B_{ij,i}\omega_i +B_{ij,j}\omega_j, ~~
(\bar{\lambda}_i-\bar{\lambda}_j)\omega_{ij}
=\bar{B}_{ij,i}\omega_i+\bar{B}_{ij,j}\omega_j,  ~~i\ne j.
\end{equation}

To derive the exact expressions of the connection forms
$\omega_{ij}$, we shall compute out all the
quantities like $B_{ii,j}$ in terms of $\lambda_k$'s and $C_k$'s.
By the symmetry in our situation, obviously there is
$\bar{B}_{33,1}=B_{44,1},\bar{B}_{44,1}=B_{33,1}$.
Together with \eqref{41} and \eqref{equa3}, it follows
\[
(\bar{\lambda}_3-\bar{\lambda}_1)(B_{33,1}-C_1)
=(\bar{\lambda}_3-\bar{\lambda}_1)B_{31,3}
=(\lambda_3-\lambda_1)\bar{B}_{31,3}
=(\lambda_3-\lambda_1)(\bar{B}_{33,1}-\bar{C}_1).
\]
So we get
\begin{equation}\label{45}
\begin{split}
&(\lambda_4-\lambda_2)(B_{33,1}-C_1)
=(\lambda_3-\lambda_1)(B_{44,1}-\bar{C}_1),\\
&(\lambda_3-\lambda_2)(B_{44,1}-C_1)
=(\lambda_4-\lambda_1)(B_{33,1}-\bar{C}_1).
\end{split}
\end{equation}
The second equation is obtained in the similar way.

On the other hand, \eqref{42} tells us that
\[
B_{22,1}-C_1=B_{12,2}=-\bar{B}_{12,2}
=\bar{B}_{22,1}-\bar{C}_1=B_{11,1}-\bar{C}_1.
\]
Thus $B_{11,1}+B_{22,1}=C_1+\bar{C}_1$.
Note that $\sum_i\lambda_i=\sum_iB_{ii}=0$ and
$\sum_iB_{ii}^2=\frac{n-1}{n}$ imply
\begin{equation}\label{46}
\sum_{i=1}^4 B_{ii,k}=0,~~
\sum_{i=1}^4 \lambda_i B_{ii,k}=0,~~\forall ~k.
\end{equation}
In particular,
\begin{equation}\label{47}
B_{33,1}+B_{44,1}=-(B_{11,1}+B_{22,1})=-C_1-\bar{C}_1.
\end{equation}
Eliminating $B_{33,1},B_{44,1}$ from \eqref{47}\eqref{45}
(keep in mind $\sum_i\lambda_i=0$) yields
$\lambda_2 C_1=\lambda_1 \bar{C}_1$. Because
$\lambda_1,\lambda_2$ could not be zero at the same time
($\lambda_1\ne\pm\lambda_2$), we may denote
\begin{equation}\label{48}
\Delta_1:=\frac{C_1}{\lambda_1}=\frac{\bar{C}_1}{\lambda_2}.
\end{equation}
In case that $\lambda_1=0\ne\lambda_2$, there must be $C_1=0$,
and we need only to take $\Delta_1=\frac{\bar{C}_1}{\lambda_2}$
which is well-defined.

Putting \eqref{48} into \eqref{45}\eqref{47} solves
$B_{33,1},B_{44,1}$. Then by \eqref{46} we get the
complete solution:
\begin{equation}\label{49}
\begin{split}
&B_{11,1}
=\frac{\lambda_2\lambda_3+\lambda_2\lambda_4-\lambda_3^2-\lambda_4^2}
{\lambda_1-\lambda_2}\Delta_1,~
B_{33,1}=\lambda_3 \Delta_1,\\
&B_{22,1}
=\frac{\lambda_1\lambda_3+\lambda_1\lambda_4-\lambda_3^2-\lambda_4^2}
{\lambda_2-\lambda_1}\Delta_1,~
B_{44,1}=\lambda_4 \Delta_1.
\end{split}
\end{equation}
Similarly there are:
\begin{equation}\label{50}
\begin{split}
&B_{11,2}
=\frac{\lambda_2\lambda_3+\lambda_2\lambda_4-\lambda_3^2-\lambda_4^2}
{\lambda_1-\lambda_2}\Delta_2,~
B_{33,2}=\lambda_3 \Delta_2,\\
&B_{22,2}
=\frac{\lambda_1\lambda_3+\lambda_1\lambda_4-\lambda_3^2-\lambda_4^2}
{\lambda_2-\lambda_1}\Delta_2,~
B_{44,1}=\lambda_4 \Delta_2,\\
&B_{11,3}=\lambda_1 \Delta_3,~
B_{33,3}
=\frac{\lambda_4\lambda_1+\lambda_4\lambda_2-\lambda_1^2-\lambda_2^2}
{\lambda_3-\lambda_4}\Delta_3,\\
&B_{22,3}=\lambda_2 \Delta_3,~
B_{44,3}
=\frac{\lambda_3\lambda_1+\lambda_3\lambda_2-\lambda_1^2-\lambda_2^2}
{\lambda_4-\lambda_3}\Delta_3,\\
&B_{11,4}=\lambda_1 \Delta_4,~
B_{33,4}
=\frac{\lambda_4\lambda_1+\lambda_4\lambda_2-\lambda_1^2-\lambda_2^2}
{\lambda_3-\lambda_4}\Delta_4,\\
&B_{22,4}=\lambda_2 \Delta_4,~
B_{44,4}
=\frac{\lambda_3\lambda_1+\lambda_3\lambda_2-\lambda_1^2-\lambda_2^2}
{\lambda_4-\lambda_3}\Delta_4,
\end{split}
\end{equation}
where
\[
\Delta_2:=\frac{C_2}{\lambda_2}=\frac{\bar{C}_2}{\lambda_1},~
\Delta_3:=\frac{C_3}{\lambda_3}=\frac{\bar{C}_3}{\lambda_4},~
\Delta_4:=\frac{C_4}{\lambda_4}=\frac{\bar{C}_4}{\lambda_3}.
\]
Now the connection forms $\omega_{ij}$ could be determined.
Since $\lambda_1-\lambda_2\ne 0$, by \eqref{44}\eqref{equa3}
\eqref{49}\eqref{50} and $C_1=\lambda_1\Delta_1,
C_2=\lambda_2\Delta_2, \sum_i\lambda_i=0$, we get
\[
\omega_{12}
=\frac{B_{11,2}-C_2}{\lambda_1-\lambda_2}\omega_1
+\frac{B_{22,1}-C_1}{\lambda_1-\lambda_2}\omega_2
=I_{12}(\Delta_2 \omega_1-\Delta_1 \omega_2),~~
I_{12}:=-\frac{2\lambda_1\lambda_2+\lambda_3^2+\lambda_4^2}
{(\lambda_1-\lambda_2)^2}.
\]
Similarly, there is
\[
\omega_{34}
=I_{34}(\Delta_4 \omega_3-\Delta_3 \omega_4),~~
I_{34}:=-\frac{2\lambda_3\lambda_4+\lambda_1^2+\lambda_2^2}
{(\lambda_3-\lambda_4)^2}.
\]
Other connection forms are found in the same way, yet much easier:
\begin{equation*}
\begin{split}
&\omega_{13}=\Delta_3 \omega_1-\Delta_1 \omega_3,~~
\omega_{24}=\Delta_4 \omega_2-\Delta_2 \omega_4, \\
&\omega_{14}=\Delta_4 \omega_1-\Delta_1 \omega_4,~~
\omega_{23}=\Delta_3 \omega_2-\Delta_2 \omega_3.
\end{split}
\end{equation*}
Finally, by the formula
$d\omega_{ij}-\sum_l\omega_{il}\wedge\omega_{jl}
=-\frac{1}{2}R_{ijkl}~\omega_k\wedge\omega_l,$
the sectional curvatures are computed out:
\begin{equation*}
\begin{split}
&\tfrac{1}{2}R_{1313}
=-E_1(\Delta_1)-E_3(\Delta_3)
+\Delta_1^2+\Delta_3^2
+I_{12}\Delta_2^2+I_{34}\Delta_4^2,\\
&\tfrac{1}{2}R_{2424}
=-E_2(\Delta_2)-E_4(\Delta_4)
+\Delta_2^2+\Delta_4^2
+I_{12}\Delta_1^2+I_{34}\Delta_3^2,\\
&\tfrac{1}{2}R_{1414}
=-E_1(\Delta_1)-E_4(\Delta_4)
+\Delta_1^2+\Delta_4^2
+I_{12}\Delta_2^2+I_{34}\Delta_3^2,\\
&\tfrac{1}{2}R_{2323}
=-E_2(\Delta_2)-E_3(\Delta_3)
+\Delta_2^2+\Delta_3^2
+I_{12}\Delta_1^2+I_{34}\Delta_4^2.
\end{split}
\end{equation*}
Here $E_i(\Delta_j)$ is understood as the action of tangent
vector $E_i$ on the function $\Delta_j$, and $\Delta_i^2$
is the square of $\Delta_i$.
As a corollary,
\[
R_{1313}+R_{2424}-R_{1414}-R_{2323}=0.
\]
But on the other hand, \eqref{equa4} implies
$R_{ijij}=\lambda_i\lambda_j+A_{ii}+A_{jj}$ when $i\ne j$.
Substitute this into the final result above, we find
\[
R_{1313}+R_{2424}-R_{1414}-R_{2323}
=(\lambda_1-\lambda_2)(\lambda_3-\lambda_4)=0.
\]
This contradicts our assumption
$\lambda_1\ne\pm\lambda_2, \lambda_3\ne\pm\lambda_4$.
Thus we have proved that the possibilities other than
$B=\pm \bar{B}$ could not happen. This completes the proof to Proposition \ref{prop63}.
\end{proof}

\vskip 0.5 cm
\section{Deformable hypersurfaces with one principal curvature of multiplicity $n-1$}

In this section and the next one we
make use of the following convention on the range of indices:
  $$1\leq i,j,k \leq n;\ 3\leq \alpha,\beta,\gamma \leq n.$$
\begin{PROPOSITION}\label{pro32}
Let $f,\bar{f}: M^n\rightarrow R^{n+1}(n\geq 4)$ be two
hypersurfaces without umbilics.
Suppose that their M\"{o}bius metrics are equal,
and one principal curvature of $B$ has multiplicity
$n-1$ everywhere (this means that $(M^n,g)$ is
conformally flat).
Then either $f(M^n)$ is M\"{o}bius congruent to $\bar{f}(M^n)$,
or $f(M^n)$ has constant M\"{o}bius curvature.
\end{PROPOSITION}
\begin{proof}  Since one of principal
curvatures of $B$ has multiplicity $n-1$ everywhere,  from the
algebraic Theorem \ref{th1}, locally we can choose an orthonormal basis
$\{E_1,\cdots,E_n\}$ with respect to $g$ such that
\begin{itemize}
\item[]
$\{\bar{B}_{ij}\}=\text{diag}\bigl( \Big[\begin{smallmatrix}
\bar{B}_{11}& \bar{B}_{12}\\
\bar{B}_{21}& \bar{B}_{22}
\end{smallmatrix}\Big],
\bar{\mu},\cdots,\bar{\mu}\bigr), \{B_{ij}\} =\text{diag}(\lambda,\mu,
\cdots,\mu),$ .
\end{itemize}
where $\lambda\neq\mu$. From (\ref{equa6}) we have $\lambda+(n-1)\mu=0$ and
$\lambda^2+(n-1)\mu^2=\frac{n-1}{n}$. So
\begin{equation*}
\{B_{ij}\} =\text{diag}\left(\frac{n-1}{n},\frac{-1}{n}, \cdots,\frac{-1}{n}\right).
\end{equation*}

Since $\bar{\mu}=\pm\mu$, up to change of the normal direction
we may assume $\bar{\mu}=\mu=\frac{-1}{n}$.
Apply formula \eqref{equa6} to $\bar{B}_{ij}$. We know that the
sub-matrices
\[
\begin{pmatrix}
\bar{B}_{11}& \bar{B}_{12}\\
\bar{B}_{21}& \bar{B}_{22}
\end{pmatrix}
\quad\text{and}\quad
\begin{pmatrix}
\frac{n-1}{n} & 0\\
0 & \frac{-1}{n}
\end{pmatrix}
\]
have equal traces and equal norms, and $B,\bar{B}$
share equal eigenvalues. At any point of $M^n$ there exists a suitable $P=\begin{pmatrix}
\cos\theta & \sin\theta\\
-\sin\theta & \cos\theta
\end{pmatrix}\in SO(2)$
such that
\begin{equation}\label{2b2}
\begin{pmatrix}
\bar{B}_{11}& \bar{B}_{12}\\
\bar{B}_{21}& \bar{B}_{22}
\end{pmatrix}
= P^{-1}\begin{pmatrix}
\frac{n-1}{n} & 0\\
0 & \frac{-1}{n}
\end{pmatrix}P=
\begin{pmatrix}
\cos^2\theta-\frac{1}{n} & \cos\theta\sin\theta \\
\cos\theta\sin\theta  & \sin^2\theta-\frac{1}{n}
\end{pmatrix}.
\end{equation}

We summarize some intermediate results as
\begin{lemma}\label{lemma54}
For hypersurface $f$, the coefficients of tensor $B,\nabla B,C$
under the orthonormal basis $\{E_1,\cdots,E_n\}$ satisfy
\begin{equation*}
\begin{split}
&B_{1j,j}=-C_1, j>1; ~~~~\text{otherwise,} B_{ij,k}=0.\\
&C_k=0,k>1;~~~~~\omega_{1j}=-C_1\omega_j, j>1.\\
&R_{1i1i}-C_{1,1}+C_1^2=0, j>1;~~~~R_{1iji}-C_{1,j}=0,j>1.
\end{split}
\end{equation*}
where $\{\omega_1,\cdots,\omega_n\}$ are the dual basis,
and $\{\omega_{ij}\}$ are its connection forms.
\end{lemma}
\begin{proof} From
$dB_{ij}+\sum_kB_{kj}\omega_{ki}+B_{ik}\omega_{kj}=\sum_kB_{ij,k}\omega_k$
and (\ref{equa3}) we get the first four equalities.
From
$d\omega_{1i}-\sum_k\omega_{1k}\wedge\omega_{ki}
=-\frac{1}{2}\sum_{kl}R_{1ikl}\omega_k\wedge\omega_l$
and invoking the proved equalities we get the equalities on the curvature tensor.
\end{proof}
In order to prove Proposition (\ref{pro32}) We have to consider the following two cases:

{\bf Case I}, $\bar{B}_{12}\equiv 0$;

{\bf Case II}, $\bar{B}_{12}\neq 0$.

First we consider Case I. Since
$\bar{B}_{12}=-\cos\theta\sin\theta\equiv 0$, so $\sin\theta=0$ or
$\cos\theta=0$. If $\sin\theta=0$, then $\bar{B}=B$, thus  $f$ is
M\"{o}bius congruent to $\bar{f}$.  Next we assume that
$\cos\theta=0$. From (\ref{2b2}) we get
\begin{equation*}
(\bar{B}_{ij})=\text{diag}(\mu,\lambda,\mu,\cdots,\mu)
=\text{diag}\left(\frac{-1}{n},\frac{n-1}{n},\frac{-1}{n},
\cdots,\frac{-1}{n}\right).
\end{equation*}

For hypersurface $\bar{f}:M^n\rightarrow S^{n+1}$, since
$g=\bar{g}$, we may use the same dual basis $\{\omega_i\}$ with the same connection forms. Similar to Lemma~\ref{lemma54} we get
\begin{equation*}
\begin{split}
&\bar{B}_{2i,i}=-\bar{C}_2;~~\text{otherwise} ~~\bar{B}_{ij,k}=0;\\
&\omega_{1j}=-\bar{C}_2\omega_j,j\neq 2;~~\bar{C}_i=0,i\neq2.
\end{split}
\end{equation*}
Compared with Lemma~\ref{lemma54} we see
$-C_1\omega_2=\omega_{12}=-\omega_{21}=\bar{C}_2\omega_1.$
Therefore, $C_1=\bar{C}_2=0$, and the M\"{o}bius forms of both
$f$ and $\bar{f}$ vanish:
\begin{equation*}
C=0,~~~\bar{C}=0.
\end{equation*}
Thus both $f$ and $\bar{f}$ are M\"{o}bius isoparametric
hypersurfaces with two distinct principal curvatures. Since
$\omega_{1i}=\omega_{2i}=0$ and $R_{1212}=0$, from \cite{lw},
$f$ and $\bar{f}$ are M\"{o}bius equivalent to the
circular cylinder $S^1(1)\times R^{n-1}\subset R^{n+1}$,
hence congruent to each other. This completes the
proof to Proposition \ref{pro32} for Case I. (In particular this is not a M\"{o}bius deformable case.)

Next we consider Case II, $\bar{B}_{12}=-\sin\theta\cos\theta\neq 0$.\\
Since $B_{ij}=\frac{-1}{n}\delta_{ij}, 2\leq i,j\leq n,
\bar{B}_{\alpha\beta}=\frac{-1}{n}\delta_{\alpha\beta}$, We can
rechoose $\{E_2,\cdots,E_n\}$ such that
\begin{equation*}
(A_{ij})=\left(
  \begin{array}{ccccc}
    A_{11} & A_{12} & A_{13} & \cdots & A_{1n} \\
    A_{21} & a_2 & 0 & \cdots & 0 \\
    A_{31}& 0 & a_3& \cdots & 0 \\
    \vdots & \vdots & \vdots& \ddots & \vdots \\
   A_{n1} & 0 & 0 & \cdots &a_n \\
  \end{array}
\right),(\bar{A}_{ij})=\left(\begin{array}{cccccc}
    \bar{A}_{11} & \bar{A}_{12} & \bar{A}_{13}& \bar{A}_{14} & \cdots & \bar{A}_{1n} \\
    \bar{A}_{21} & \bar{A}_{22} & \bar{A}_{23} & \bar{A}_{24} & \cdots & \bar{A}_{2n} \\
    \bar{A}_{31}& \bar{A}_{32} & \bar{a}_3&0& \cdots & 0 \\
    \bar{A}_{41}&\bar{A}_{42}& 0& \bar{a}_4 &\cdots& 0\\
    \vdots & \vdots & \vdots&\vdots& \ddots & \vdots \\
   \bar{A}_{n1} &\bar{A}_{n2}& 0 & 0 & \cdots &\bar{a}_n \\
  \end{array}
\right).
\end{equation*}
Noting that Lemma \ref{lemma54} holds under this basis, using
$R_{1i1i}-C_{1,1}+C_1^2=0$ and (\ref{equa4}) we get that
$$a_2=a_3=\cdots=a_n,\bar{a}_3=\cdots=\bar{a}_n.$$
In formula (\ref{equa4}), Let $i=2,  k=\alpha,j=l=1$ and $i=k=\alpha,k=l=\beta$ we
get that
$$\bar{A}_{\alpha 2}=\bar{A}_{2\alpha}=0,a_2=a_3=\cdots=a_n=\bar{a}_3=\cdots=\bar{a}_n.$$
Thus we have
\begin{equation}\label{2bab}
(A_{ij})=\left(
  \begin{array}{ccccc}
    A_{11} & A_{12} & A_{13} & \cdots & A_{1n} \\
    A_{21} & a_2 & 0 & \cdots & 0 \\
    A_{31}& 0 & a_2& \cdots & 0 \\
    \vdots & \vdots & \vdots& \ddots & \vdots \\
   A_{n1} & 0 & 0 & \cdots &a_2 \\
  \end{array}~~
\right),\bar{(A)}_{ij}=\left(
  \begin{array}{ccccc}
    \bar{A}_{11} & \bar{A}_{12} & \bar{A}_{13} & \cdots & \bar{A}_{1n} \\
    \bar{A}_{21} & \bar{A}_{22} & 0 & \cdots & 0 \\
    \bar{A}_{31}& 0 & a_2& \cdots & 0 \\
    \vdots & \vdots & \vdots& \ddots & \vdots \\
   \bar{A}_{n1} & 0 & 0 & \cdots &a_2 \\
  \end{array}
\right).
\end{equation}
Since $\bar{B}_{1\alpha}=0,\bar{B}_{2\alpha}=0,
\bar{B}_{\alpha\beta}=\frac{-1}{n}\delta_{\alpha\beta}$,
we do covariant differentiation to find
\begin{equation}\label{2b11}
\begin{split}
&(\bar{B}_{11}+\frac{1}{n})\omega_{1\alpha}+\bar{B}_{12}
\omega_{2\alpha}=\sum{}_k\bar{B}_{1\alpha,k}\omega_k;\\
&\bar{B}_{12}\omega_{1\alpha}+(\bar{B}_{22}+\frac{1}{n})
\omega_{2\alpha}=\sum{}_k\bar{B}_{2\alpha,k}\omega_k;\\
&\bar{B}_{\alpha\beta,k}=0,~~\forall~\alpha,\beta,k.
\end{split}
\end{equation}
Using
$E_{\beta}(\bar{B}_{\alpha\alpha})=\bar{B}_{\alpha\alpha,\beta}$,
(\ref{equa3}) and (\ref{2b11}), we get that
\begin{equation}\label{2b5}
\bar{C}_{\alpha}=0,\bar{B}_{1\alpha,\alpha}=-\bar{C}_1,\bar{B}_{2\alpha,\alpha}=-\bar{C}_2,\bar{B}_{1\alpha,\beta}=\bar{B}_{2\alpha,\beta}=0,\alpha\neq\beta.
\end{equation}
Thus from (\ref{2b2}) and (\ref{2b11}) we have
\begin{equation}\label{2b6}
\begin{split}
&\cos^2\theta\omega_{1\alpha}+\sin\theta\cos\theta\omega_{2\alpha}=\bar{B}_{1\alpha,1}\omega_1+\bar{B}_{1\alpha,2}\omega_2+\bar{B}_{1\alpha,\alpha}\omega_{\alpha};\\
&\sin\theta\cos\theta\omega_{1\alpha}+\sin^2\theta\omega_{2\alpha}=\bar{B}_{2\alpha,1}\omega_1+\bar{B}_{2\alpha,2}\omega_2+\bar{B}_{2\alpha,\alpha}\omega_{\alpha}.
\end{split}
\end{equation}
Using $E_{\beta}(\bar{B}_{11})=\bar{B}_{11,\alpha}$,
$\bar{B}_{11}+\bar{B}_{22}-\frac{n-2}{n}=0$ and (\ref{2b6}) we
derive
\begin{equation}\label{2b7}
\begin{split}
&\bar{B}_{1\alpha,1}=\bar{B}_{2\alpha,2}=\bar{B}_{1\alpha,2}=0,\sin\theta\bar{C}_1=\cos\theta\bar{C}_2,\\
&\omega_{2\alpha}=\frac{\cos^2\theta
C_1-\bar{C}_1}{\cos\theta\sin\theta}\omega_{\alpha}.
\end{split}
\end{equation}
 Using
$d\bar{B}_{2\alpha,2}+\sum_k\bar{B}_{k\alpha,2}\omega_{k2}+\sum_k\bar{B}_{2k,2}\omega_{k\alpha}+\sum_k\bar{B}_{2\alpha,k}\omega_{k2}=\sum_k\bar{B}_{2\alpha,2k}\omega_k$
,(\ref{2b5}) and (\ref{2b7}),  we get
\[
\bar{B}_{2\alpha,21}=0.
\]
Similarly we can get $\bar{B}_{2\alpha,12}=0.$
Using Ricci identity
$\bar{B}_{2\alpha,21}-\bar{B}_{2\alpha,12}=\sum_k\bar{B}_{k\alpha}R_{k221}+\sum_k\bar{B}_{2k}R_{k\alpha21}$
and $R_{1\alpha12}=A_{2\alpha}=\bar{A}_{2\alpha}=0$ we get $R_{2\alpha21}=0.$ By \eqref{equa4} this implies
\begin{gather*}
A_{1\alpha}=\bar{A}_{1\alpha}=0,\\
\{A_{ij}\}=\text{diag}\bigl( \Big[\begin{smallmatrix}
A_{11}& A_{12}\\
A_{21}& a_2
\end{smallmatrix}\Big],
a_2,\cdots,a_2\bigr),~~
\{\bar{A}_{ij}\}=\text{diag}\bigl( \Big[\begin{smallmatrix}
\bar{A}_{11}& \bar{A}_{12}\\
\bar{A}_{21}& \bar{A}_{22}
\end{smallmatrix}\Big],
a_2,\cdots,a_2\bigr).
\end{gather*}
From \eqref{equa6} we have
\[
\bar{B}_{11}+\bar{B}_{22}=-\frac{n-2}{n},
\bar{B}^2_{11}+\bar{B}^2_{22}+2\bar{B}^2_{12}
=\frac{n^2-2n+2}{n^2}.
\]
Using the above identity and
$E_k(\bar{B}_{ij})+\sum_l\bar{B}_{lj}\omega_{li}(E_k)+\sum_l\bar{B}_{il}\omega_{lj}(E_k)=\bar{B}_{ij,k}$,
we get
\begin{equation*}
\begin{split}
&\bar{B}_{11,1}+\bar{B}_{22,1}=0, \bar{B}_{11,2}+\bar{B}_{22,2}=0,\\
&\bar{B}_{11}\bar{B}_{11,1}+\bar{B}_{22}\bar{B}_{22,1}+2\bar{B}_{12}\bar{B}_{12,1}=0,\\
&\bar{B}_{11}\bar{B}_{11,2}+\bar{B}_{22}\bar{B}_{22,2}+2\bar{B}_{12}\bar{B}_{12,2}=0.
\end{split}
\end{equation*}
From the above equation, \eqref{equa6},
$\bar{B}_{11,2}=\bar{B}_{12,1}+\bar{C}_2,$ and $\bar{B}_{22,1}=\bar{B}_{12,2}+\bar{C}_1,$ we derive
\begin{equation}\label{2b10}
\begin{split}
&\bar{B}_{11,2}=2\cos\theta\sin\theta\bar{C}_1,
\bar{B}_{22,1}=2\cos\theta\sin\theta\bar{C}_2,\\
&\bar{B}_{12,1}=(\cos^2\theta-\sin^2\theta)\bar{C}_2,\bar{B}_{12,2}=(\sin^2\theta-\cos^2\theta)\bar{C}_1.
\end{split}
\end{equation}
Since
$E_1(\bar{B}_{12})=E_1(\cos\theta\sin\theta)=(\cos^2\theta-\sin^2\theta)E_1(\theta)$
and $E_1(\bar{B}_{12})=\bar{B}_{12,1}$, from (\ref{2b10}) we get
$E_1(\theta)=\bar{C}_2$. Similarly we have
$E_2(\theta)=C_1-\bar{C}_1$, thus we have
\begin{equation}\label{2b12}
E_1(\theta)=\bar{C}_2, E_2(\theta)=C_1-\bar{C}_1,
E_{\alpha}(\theta)=0.
\end{equation}
Combining Lemma 5.4 and (\ref{2b7}) we have
\begin{equation*}
\begin{split}
&d\omega_1=0, d\omega_2=-C_1\omega_1\wedge\omega_2,\\
&d\omega_{\alpha}=C_1\omega_1\wedge\omega_{\alpha}+\frac{\cos^2\theta
C_1-\bar{C}_1}{\cos\theta\sin\theta}\omega_2\wedge\omega_{\alpha}+\sum_{\beta}\omega_{\beta}\wedge\omega_{\beta\alpha}.
\end{split}
\end{equation*}
Therefore we have
$$[E_1,E_2]=C_1E_2.$$

Using $d\bar{C}_1+\bar{C}_2\omega_{21}=\sum_k\bar{C}_{1,k}\omega_k$ and $d\bar{C}_1+\bar{C}_2\omega_{21}=\sum_k\bar{C}_{1,k}\omega_k$, we have
\begin{equation}\label{cc1}
\begin{split}
&E_1(\bar{C}_1)=\bar{C}_{1,1},E_2(\bar{C}_1)=\bar{C}_{1,2}-C_1\bar{C}_2,\\
&E_1(\bar{C}_2)=\bar{C}_{2,1}, E_2(\bar{C}_2)=\bar{C}_{2,2}+C_1\bar{C}_1.
\end{split}
\end{equation}

Using $[E_1,E_2](\theta)=C_1E_2(\theta)$, (\ref{2b12}) and
(\ref{cc1}) we get that
\begin{equation}\label{2b14}
\bar{C}_{1,1}+\bar{C}_{2,2}=C_{1,1}-C_1^2=R_{1\alpha1\alpha}.
\end{equation}
Combining  $\sin\theta\bar{C}_1=\cos\theta\bar{C}_2$, (\ref{2b12}) and (\ref{cc1}), we obtain that
$$\bar{C}_{1,2}=\frac{\cos\theta}{\sin\theta}(\bar{C}_1^2+\bar{C}_2^2+\bar{C}_{2,2}),\bar{C}_{2,1}=\frac{\sin\theta}{\cos\theta}(\bar{C}_1^2+\bar{C}_2^2+\bar{C}_{1,1})$$
From above formula and (\ref{equa3}), we have
\begin{equation}\label{r121}
R_{1\alpha2\alpha}=\frac{\cos^2\theta-\sin^2\theta}{\sin\theta\cos\theta}(\bar{C}_1^2+\bar{C}_2^2)
+\frac{\cos\theta}{\sin\theta}\bar{C}_{2,2}-\frac{\sin\theta}{\cos\theta}\bar{C}_{1,1}.
\end{equation}

Compute the covariant differentiation of $\bar{B}_{1\alpha,\alpha},\bar{B}_{1\alpha,1}$.
By \eqref{2b5} and the Ricci identity,
\begin{equation}\label{r123}
R_{1\alpha2\alpha}+\frac{\sin\theta}{\cos\theta}R_{2\alpha2\alpha}
+\frac{\sin^2\theta-\cos^2\theta}{\sin^2\theta\cos^2\theta}\bar{C}_1\bar{C}_2=\frac{\bar{C}_{2,2}}{\cos\theta\sin\theta}.
\end{equation}
Similarly using Ricci identity $\bar{B}_{2\alpha,2\alpha}-\bar{B}_{2\alpha,\alpha2}
=\sum_k\bar{B}_{k\alpha}R_{k22\alpha}
+\sum_k\bar{B}_{2k}R_{k\alpha2\alpha}$
we get
\begin{equation}\label{2b15}
R_{1\alpha2\alpha}+\frac{\cos\theta}{\sin\theta}R_{1\alpha1\alpha}
+\frac{\cos^2\theta-\sin^2\theta}{\sin^2\theta\cos^2\theta}\bar{C}_1\bar{C}_2=\frac{\bar{C}_{1,1}}{\cos\theta\sin\theta}.
\end{equation}
Sum (\ref{r123}) and (\ref{2b15}). Using (\ref{2b14}) we get that
\begin{equation}\label{r124}
2R_{1\alpha2\alpha}+\frac{\sin\theta}{\cos\theta}(R_{2\alpha2\alpha}-R_{1\alpha1\alpha})=0.
\end{equation}
Note that $\sin\theta\bar{C}_1=\cos\theta\bar{C}_2$, combining (\ref{r121}),(\ref{r123}) and (\ref{2b15}), we obtain that
\begin{equation}\label{r125}
2R_{1\alpha2\alpha}-
\frac{\cos\theta}{\sin\theta}
(R_{2\alpha2\alpha}-R_{1\alpha1\alpha})=0.
\end{equation}
From (\ref{r124}) and (\ref{r125}), we get that
\begin{equation}\label{rr}
R_{1\alpha1\alpha}-R_{2\alpha2\alpha}=0,~R_{1\alpha2\alpha}=0.
\end{equation}
Therefore from \eqref{equa4} we get  $R_{2\alpha2\alpha}=R_{\alpha\beta\alpha\beta}$. Hence
$R_{ijkl}=R_{2\alpha2\alpha}(\delta_{ik}\delta_{jl}
-\delta_{il}\delta_{ij}).$
By Schur's theorem $(M^n,g)$ is of constant curvature.
This completes the proof.
\end{proof}

\section{Deformable hypersurfaces with one principal curvature
of multiplicity $n-2$}
This section is devoted to the proof of the following
\begin{PROPOSITION}\label{prop-n-2}
Let $f,\bar{f}: M^n\rightarrow R^{n+1}~~(n\geq 4)$ be two
immersed hypersurfaces without
umbilics, whose principal curvatures have constant multiplicities. Suppose their M\"{o}bius metrics are equal,
and one of principal curvatures of $B$ has multiplicity
$n-2$ everywhere.
Assume that $f(M^n)$ is NOT M\"{o}bius congruent to
$\bar{f}(M^n)$. Then it must be either of the following three cases:

 (1) $f(M^n)$ is congruent to part of $L^2\times R^{n-2}$ and
$\bar{f}(M^n)$ is congruent to part of
$\bar{L}^2\times R^{n-2}$, where $L^2$ and $\bar{L}^2$
are a pair of isometric Bonnet surface in $R^3$.

 (2) $f(M^n)$ is congruent to part of $CL^2\times R^{n-3}$
 where $CL^2\subset R^4$ is a cone over $L^2\subset S^3$, and
$\bar{f}(M^n)$ is congruent to part of $C\bar{L}^2\times R^{n-3}$. $L^2$ and $\bar{L}^2$ form a Bonnet pair in $S^3$.

(3) $f(M^n)$ is a rotation hypersurfaces over
$L^2\subset R^3_+$, and $\bar{f}(M^n)$ is a rotation hypersurfaces over $\bar{L}^2\subset R^3_+$, where $L^2$ and
$\bar{L}^2$ form a Bonnet pair in hyperbolic half-space $R^3_+$.
\end{PROPOSITION}
Recall that we have adopted the following convention on
the range of indices as the last section:
\[
1\leq i,j,k \leq n;\ 3\leq \alpha,\beta,\gamma \leq n.
\]
Since one of principal
curvatures of $B$ has multiplicity $(n-2)$ everywhere,
by Theorem \ref{th1} we can assume without loss of generality
that there exists a local orthonormal basis $\{E_1,\cdots,E_n\}$
for $(M^n,g)$ which is shared by $f,\bar{f}$, such that
\begin{eqnarray}\label{bb1}
\{B_{ij}\}=\text{diag}(\lambda_1,\lambda_2,\mu,\cdots,\mu);
\{\bar{B}_{ij}\}=\text{diag}\Bigl(\left[
    \begin{smallmatrix}
      \bar{B}_{11} & \bar{B}_{12}\\
      \bar{B}_{21} &\bar{B}_{22}
    \end{smallmatrix}
  \right],
 \mu,\cdots,\mu \Bigr),
\end{eqnarray}
where $\lambda_1\neq\mu, \lambda_2\neq\mu.$ By the identities \eqref{equa6} we know that the sub-matrices
\[
\left(\begin{smallmatrix}
\lambda_1 & 0\\
0 & \lambda_2
\end{smallmatrix}\right)
\quad\text{and}\quad
\left(\begin{smallmatrix}
\bar{B}_{11}& \bar{B}_{12}\\
\bar{B}_{21}& \bar{B}_{22}
\end{smallmatrix}\right)
\]
have equal traces and equal norms. Hence $B,\bar{B}$
share the same eigenvalues. Therefore, the highest multiplicity of principal curvatures of $\bar{f}$ is also $n-2$.
Denote $\bar{B}_{11}=\bar{\lambda}_1,~\bar{B}_{22}=\bar{\lambda}_2$.

We assert that $\lambda_1\ne\lambda_2$ and
$\lambda_1\ne\bar{\lambda}_1$ on an open dense subset of $M$.
Otherwise, suppose $\lambda_1=\lambda_2$ on an open subset.
Then the 2 by 2 sub-matrices share the same eigenvalues
$\lambda_1=\lambda_2$, hence both be scalar matrix
$\text{diag}(\lambda_1,\lambda_1)$. We get $B=\bar{B}$
on an open subset. So these two hypersurfaces are
M\"obius equivalent. Contradiction.
If $\lambda_1=\bar{\lambda}_1$ on an open subset
we will get a similar contradiction.

From now on, without loss of generality we assume that on $M^n$
\begin{equation}\label{eq-ne}
\lambda_1\ne\mu, ~\lambda_2\ne\mu, ~\lambda_1\ne\lambda_2,~
 \lambda_1\ne\bar{\lambda}_1.
\end{equation}
By \eqref{equa6} they satisfy
\begin{equation}\label{H1}
\begin{split}
&\lambda_1+\lambda_2+(n-2)\mu=\bar{\lambda}_1+\bar{\lambda}_2+(n-2)\mu=0,\\
&\lambda_1^2+\lambda_2^2+(n-2)\mu^2=\bar{\lambda}_1^2+\bar{\lambda}_2^2+2\bar{B}_{12}^2+(n-2)\mu^2=\frac{n-1}{n}.
\end{split}
\end{equation}
Choose the dual basis $\{\omega_i\}$ with connection forms
$\omega_{ij}$ satisfying
$d\omega_i=\sum_k\omega_{ik}\wedge\omega_k,~
\omega_{ij}=-\omega_{ji}$.
It follows from $dB_{ij}+\sum_kB_{kj}\omega_{ki}
+\sum_kB_{ik}\omega_{kj}=\sum_kB_{ij,k}\omega_k$ that
\begin{equation}\label{01}
\begin{split}
&0=B_{\alpha\beta,i}=B_{1\alpha,\beta}
=B_{2\alpha,\beta}, ~~\alpha\ne\beta;\\
&(\lambda_1-\mu)\omega_{1\alpha} =B_{1\alpha,1}\omega_1
+B_{1\alpha,2}\omega_2
+B_{1\alpha,\alpha}\omega_{\alpha};\\
&(\lambda_2-\mu)\omega_{2\alpha} =B_{2\alpha,1}\omega_1
+B_{2\alpha,2}\omega_2
+B_{2\alpha,\alpha}\omega_{\alpha} ;\\
&(\lambda_1-\lambda_2)\omega_{12} =\sum{}_kB_{12,k}\omega_k.
\end{split}
\end{equation}
Using
$d\bar{B}_{ij}+\sum_k\bar{B}_{kj}\omega_{ki}+\sum_k\bar{B}_{ik}\omega_{kj}
=\sum_k\bar{B}_{ij,k}\omega_k$ in a similar way we obtain
\begin{equation}\label{02}
\begin{split}
&0=\bar{B}_{\alpha\beta,i}=\bar{B}_{1\alpha,\beta}
=\bar{B}_{2\alpha,\beta}, ~~\alpha\ne\beta;\\
&(\bar{\lambda}_1-\mu)\omega_{1\alpha}+\bar{B}_{12}\omega_{2\alpha}
=\bar{B}_{1\alpha,1}\omega_1+\bar{B}_{1\alpha,2}\omega_2
+\bar{B}_{1\alpha,\alpha}\omega_{\alpha};\\
&(\bar{\lambda}_2-\mu)\omega_{2\alpha}+\bar{B}_{12}\omega_{1\alpha}
=\bar{B}_{2\alpha,1}\omega_1+\bar{B}_{2\alpha,2}\omega_2
+\bar{B}_{2\alpha,\alpha}\omega_{\alpha};\\
&d\bar{B}_{12}+(\bar{\lambda}_1-\bar{\lambda}_2)\omega_{12}
=\sum{}_k\bar{B}_{12,k}\omega_k.
\end{split}
\end{equation}
Comparing \eqref{01} with \eqref{02} yields
\begin{equation}\label{cf1}
\begin{split}
&\bar{B}_{1\alpha,1}=
\frac{\bar{\lambda}_1-\mu}{\lambda_1-\mu}B_{1\alpha,1}
+\frac{\bar{B}_{12}}{\lambda_2-\mu}B_{2\alpha,1};~
\bar{B}_{2\alpha,1}=
\frac{\bar{\lambda}_2-\mu}{\lambda_2-\mu}B_{2\alpha,1}
+\frac{\bar{B}_{12}}{\lambda_1-\mu}B_{1\alpha,1};\\
&\bar{B}_{1\alpha,2}=
\frac{\bar{\lambda}_1-\mu}{\lambda_1-\mu}B_{1\alpha,2}
+\frac{\bar{B}_{12}}{\lambda_2-\mu}B_{2\alpha,2};~
\bar{B}_{2\alpha,2}=
\frac{\bar{\lambda}_2-\mu}{\lambda_2-\mu}B_{2\alpha,2}
+\frac{\bar{B}_{12}}{\lambda_1-\mu}B_{1\alpha,2};\\
&\bar{B}_{1\alpha,\alpha}=
\frac{\bar{\lambda}_1-\mu}{\lambda_1-\mu}B_{1\alpha,\alpha}
+\frac{\bar{B}_{12}}{\lambda_2-\mu}B_{2\alpha,\alpha};~
\bar{B}_{2\alpha,\alpha}=
\frac{\bar{\lambda}_2-\mu}{\lambda_2-\mu}B_{2\alpha,\alpha}
+\frac{\bar{B}_{12}}{\lambda_1-\mu}B_{1\alpha,\alpha}.
\end{split}
\end{equation}
Another corollary of \eqref{02},\eqref{cf1} and \eqref{equa3} is
\begin{equation}\label{04}
B_{\alpha\alpha,\beta}=C_{\beta}=\bar{C}_{\beta}
=\bar{B}_{\alpha\alpha,\beta}=B_{\beta\beta,\beta}
=\bar{B}_{\beta\beta,\beta}~, ~~\forall~\alpha\ne\beta.
\end{equation}
Taking the covariant derivatives for the identities \eqref{H1}
and invoking \eqref{04}, we have
\begin{equation}\label{05}
\begin{split}
&B_{11,\alpha}+B_{22,\alpha} =(2-n)C_{\alpha},~
\lambda_1B_{11,\alpha}
+\lambda_2B_{22,\alpha} =(2-n)\mu C_{\alpha};\\
&\bar{B}_{11,\alpha}+\bar{B}_{22,\alpha}=(2-n)\bar{C}_{\alpha},~
\bar{\lambda}_1\bar{B}_{11,\alpha}+\bar{\lambda}_2
\bar{B}_{22,\alpha} =(2-n)\mu
\bar{C}_{\alpha}-2\bar{B}_{12}\bar{B}_{12,\alpha}.
\end{split}
\end{equation}
The two equations in the first line have solution
\begin{equation}\label{06}
B_{11,\alpha}= \frac{\mu-\lambda_1}{\lambda_2-\lambda_1}
(2-n)C_{\alpha},~ B_{22,\alpha}=
\frac{\lambda_2-\mu}{\lambda_2-\lambda_1} (2-n)C_{\alpha}.
\end{equation}
Take the sum of the first and the fourth equations in \eqref{cf1}
and insert \eqref{06} into it.
By \eqref{equa3}, the first equation in \eqref{05},
and the identities \eqref{H1}, the result is as below
after simplification:
\begin{equation}\label{07}
(\lambda_1-\lambda_2) \mu \bar{B}_{12}B_{12,\alpha}
=\frac{n-1}{n}(\bar{\lambda}_1 -\lambda_1)C_{\alpha}.
\end{equation}
On the other hand,
take the difference between the second and the
third equations in \eqref{cf1}.
After simplification as before we get
\begin{equation}\label{08}
(\lambda_1-\bar{\lambda}_1 )(\lambda_1-\lambda_2) \mu
B_{12,\alpha}=\frac{n-1}{n}\bar{B}_{12}C_{\alpha}.
\end{equation}
It follows from \eqref{07}\eqref{08} that
\begin{equation}\label{09}
C_{\alpha}=\bar{C}_{\alpha}=0,~\forall~\alpha.
\end{equation}
Otherwise there will be $B_{12,\alpha}^2=-(\bar{\lambda}_1-\lambda_1)^2$
which is impossible.
As a corollary of \eqref{04}\eqref{06} and \eqref{09},
\begin{equation}\label{010}
B_{11,\alpha}=B_{22,\alpha}=
B_{\beta\beta,\alpha}=\bar{B}_{\beta\beta,\alpha}=0,~~
B_{1\alpha,1}=B_{2\alpha,2}=0, ~~\forall~\alpha,\beta.
\end{equation}
We emphasize that \eqref{08}\eqref{eq-ne} and $C_{\alpha}=0$
implies $\mu \cdot B_{12,\alpha}=0$.
Now we divide the proof into two cases.

{\bf Case I}, $B_{12,\alpha}=0$, for all $\alpha.$

{\bf Case II}, $B_{12,\alpha}\ne 0,$ for some $\alpha.$

First we consider Case I.
Since $C_{\alpha}=0, B_{12,\alpha}=0,$ from the
Reduction Theorem \ref{retheorem} we know that $f$ is
M\"{o}bius equivalent to a hypersurface given by Example
(\ref{ex31}),(\ref{ex32}) or (\ref{ex33}) when $Q=0, Q<0$ or $Q>0$, respectively.

We define $\omega_{ij}=\sum_k\Gamma^i_{jk}\omega_k$.
From \eqref{01} and the definition of $Q$ in the proof of
Theorem~\ref{retheorem}, we get that
\[
Q=2A_{\alpha\alpha}+\mu^2+(\Gamma^1_{\alpha\alpha})^2
+(\Gamma^2_{\alpha\alpha})^2.
\]
On the other hand, $R_{\alpha\beta\alpha\beta}=
\mu^2+2A_{\alpha\alpha}=\mu^2+2\bar{A}_{\alpha\alpha}$,
so $A_{\alpha\alpha}=\bar{A}_{\alpha\alpha}$.
Therefore \[Q=\bar{Q},\]
and $f,\bar{f}$ are congruent to two cylinders, or two cones,
or two rotational hypersurfaces over some surfaces in a 3-dimensional space form.
According to Remark~\ref{rem-metric} and \eqref{eq-g},
in either case they share the same metric
\[
g=\left[4H_u^2-\frac{2n}{n-1}(K_u+c)\right]
(I_u+I_{N^{n-2}(c)}).
\]
So they must share the same surface metric $I_u$, hence the same
surface curvature $K_u$, hence also the same mean curvature.
Therefore they come from a Bonnet pair in the corresponding 3-space. This finishes our proof of Proposition~\ref{prop-n-2}
in Case I.

Next we consider Case II where $B_{12,\alpha}\ne 0$ for some $\alpha$. We have the following results.
\begin{PROPOSITION}\label{prop-n-21}
Let $f,\bar{f}: M^n\rightarrow R^{n+1}~~(n\geq 4)$ be two
an immersed hypersurfaces without
umbilics, whose principal curvatures have constant multiplicities. Suppose their M\"{o}bius metrics are equal,
and one of principal curvatures of $B$ has multiplicity
$n-2$ everywhere. We can assume
that there exists a local orthonormal basis $\{E_1,\cdots,E_n\}$
for $(M^n,g)$ which is shared by $f,\bar{f}$, such that
\begin{eqnarray*}
\{B_{ij}\}=\text{diag}(\lambda_1,\lambda_2,\mu,\cdots,\mu);
\{\bar{B}_{ij}\}=\text{diag}\Bigl(\left[
    \begin{smallmatrix}
      \bar{B}_{11} & \bar{B}_{12}\\
      \bar{B}_{21} &\bar{B}_{22}
    \end{smallmatrix}
  \right],
 \mu,\cdots,\mu \Bigr),
\end{eqnarray*}
where $\lambda_1\neq\mu, \lambda_2\neq\mu.$
If $B_{12,\alpha}\neq 0$, for some $\alpha$. Then there exist an diffeomorphism $\psi:M^n\to M^n$ and a
M\"{o}bius transformation $\Phi$ such that $\Phi\circ f=\bar{f}\circ \psi:M^n\to R^{n+1}$. Moreover, $f$ is
 M\"{o}bius equivalent to the minimal hypersurface defined by
\[
x=(x_1,x_2):M^n=N^3\times H^{n-3}(-\frac{n-1}{6n})\rightarrow S^{n+1},
\]
where
\[
x_1=\frac{y_1}{y_0},x_2=\frac{y_2}{y_0},y_0\in R^+,
y_1\in R^5, y_2\in R^{n-3}.
\]
Here $y_1:N^3\rightarrow
S^4(\sqrt{\frac{6n}{n-1}})\hookrightarrow R^5$
is Cartan's minimal isoparametric hypersurface
in $S^4(\sqrt{\frac{6n}{n-1}})$ with
three principal curvatures, and
$(y_0,y_2):H^{n-3}(-\frac{n-1}{6n})\hookrightarrow R^{n-2}_1$
is the standard embedding of the hyperbolic space of
sectional curvature $-\frac{n-1}{6n}$ into
the $(n-2)$-dimensional Lorentz space with
$-y_0^2+y_2^2=\frac{6n}{n-1}$.
\end{PROPOSITION}

\begin{proof}
Since $B_{\alpha\beta}=\bar{B}_{\alpha\beta}=0$, We can assume that
\begin{equation}\label{b1235}
B_{12,3}\neq 0,~~B_{12,\alpha}=0, \alpha\neq 3.
\end{equation}
From (\ref{05}) we have
\begin{equation}\label{bw1}
\begin{split}
&\omega_{12}=\frac{-C_2}{2\lambda_1}\omega_1-\frac{C_1}{2\lambda_1}\omega_2+\frac{B_{12,3}}{2\lambda_1}\omega_3;\\
&\omega_{13}=\frac{B_{12,3}}{\lambda_1}\omega_2-\frac{C_1}{\lambda_1}\omega_3,~\omega_{23}=\frac{B_{12,3}}{\lambda_2}\omega_2-\frac{C_2}{\lambda_2}\omega_3;\\
&\omega_{1\alpha}=\frac{-C_1}{\lambda_1}\omega_{\alpha},~~\omega_{2\alpha}=\frac{-C_2}{\lambda_2}\omega_{\alpha}, \alpha>3.
\end{split}
\end{equation}
Since $C_{\alpha}=0$, using $dC_i+\sum_mC_m\omega_{mi}=\sum_mC_{i,m}\omega_m$ and (\ref{bw1}), we get
\begin{equation}\label{bc1}
\begin{split}
&C_{\alpha,\alpha}=\frac{C_2^2-C_1^2}{\lambda_1},~~C_{\alpha,k}=0,k\neq\alpha,\alpha>3;\\
&C_{3,3}=\frac{C_2^2-C_1^2}{\lambda_1},C_{3,1}=\frac{B_{12,3}C_2}{\lambda_2},C_{3,2}=\frac{B_{12,3}C_1}{\lambda_1},C_{3,\alpha}=0,\alpha>3.
\end{split}
\end{equation}
Differentiating the equations (\ref{bw1}), we get
\begin{equation}\label{wr1}
\begin{split}
\frac{-1}{2}\sum_{kl}R_{12kl}\omega_k\wedge\omega_l=\frac{-1}{2\lambda_1}\sum_m[C_{2,m}\omega_m\wedge\omega_1+C_{1,m}\omega_m\wedge\omega_2]-
3\frac{B_{12,3}C_1}{2\lambda_1^2}\omega_1\wedge\omega_3\\
+[\frac{C_1^2+C_2^2}{2\lambda_1^2}+2\frac{B_{12,3}^2}{\lambda_1^2}]\omega_1\wedge\omega_2+3\frac{B_{12,3}C_2}{2\lambda_1^2}\omega_2\wedge\omega_3+
\frac{dB_{12,3}}{2\lambda_1}\wedge\omega_3.
\end{split}
\end{equation}
\begin{equation}\label{bw2}
\begin{split}
\frac{-1}{2}\sum_{kl}R_{13kl}\omega_k\wedge\omega_l&=\frac{1}{\lambda_1}dB_{12,3}\wedge\omega_2-\frac{1}{\lambda_1}\sum_m C_{1,m}\omega_m\wedge\omega_3-
2\frac{B_{12,3}C_1}{\lambda_1^2}\omega_1\wedge\omega_2\\
&+[\frac{C_1^2+C_2^2}{2\lambda_1^2}-\frac{B_{12,3}^2}{\lambda_1^2}]\omega_1\wedge\omega_3.
\end{split}
\end{equation}
\begin{equation}\label{bw3}
\begin{split}
\frac{-1}{2}\sum_{kl}R_{23kl}\omega_k\wedge\omega_l&=\frac{1}{\lambda_1}dB_{12,3}\wedge\omega_1-\frac{1}{\lambda_2}\sum_m C_{2,m}\omega_m\wedge\omega_3+
2\frac{B_{12,3}C_2}{\lambda_1^2}\omega_1\wedge\omega_2\\
&+[\frac{C_1^2+C_2^2}{2\lambda_1^2}-\frac{B_{12,3}^2}{\lambda_1^2}]\omega_2\wedge\omega_3.
\end{split}
\end{equation}
\begin{equation}\label{bw4}
\begin{split}
\frac{-1}{2}\sum_{kl}R_{1\alpha kl}\omega_k\wedge\omega_l&=\frac{-1}{\lambda_1}\sum_m C_{1,m}\omega_m\wedge\omega_{\alpha}
+\frac{C_1^2+C_2^2}{\lambda_1^2}\omega_1\wedge\omega_{\alpha}\\
&-\frac{B_{12,3}C_2}{\lambda_1^2}\omega_3\wedge\omega_{\alpha}
-\frac{B_{12,3}}{\lambda_1}\omega_2\wedge\omega_{3\alpha}.
\end{split}
\end{equation}
\begin{equation}\label{bw5}
\begin{split}
\frac{-1}{2}\sum_{kl}R_{2\alpha kl}\omega_k\wedge\omega_l&=\frac{-1}{\lambda_1}\sum_m C_{2,m}\omega_m\wedge\omega_{\alpha}
+\frac{C_1^2+C_2^2}{\lambda_1^2}\omega_2\wedge\omega_{\alpha}\\
&-\frac{B_{12,3}C_1}{\lambda_1^2}\omega_3\wedge\omega_{\alpha}
+\frac{B_{12,3}}{\lambda_1}\omega_1\wedge\omega_{3\alpha}.
\end{split}
\end{equation}
Comparing the coefficients of $\omega_3\wedge\omega_{\alpha}$ on both sides of (\ref{bw2}), and using (\ref{equa4}) we obtain
\begin{equation}\label{ba1}
A_{1\alpha}=0,~~A_{3\alpha}=0, ~E_{\alpha}(B_{12,3})=0, ~\alpha>3.
\end{equation}
Similarly from (\ref{bw1}),(\ref{bw3}), (\ref{bw4}) and (\ref{bw5}), we have
\begin{equation}\label{ba2}
\begin{split}
&A_{2\alpha}=0,~~ A_{\alpha\beta}=0, \alpha,\beta>3, \alpha\neq\beta;\\
&\omega_{3\alpha}(E_1)=0,~\omega_{3\alpha}(E_2)=0,~\omega_{3\alpha}(E_3)=0,~\omega_{3\alpha}(E_{\beta})=0,~\alpha,\beta>3, \alpha\neq\beta;\\
&R_{1\alpha1\alpha}=-\frac{C_1^2+C_2^2}{\lambda_1^2}+\frac{C_{1,1}}{\lambda_1},R_{2\alpha2\alpha}=-\frac{C_1^2+C_2^2}{\lambda_1^2}+\frac{C_{2,2}}{\lambda_2},\alpha>3;\\
&R_{1313}=\frac{B_{12,3}^2}{\lambda_1^2}-\frac{C_1^2+C_2^2}{\lambda_1^2}+\frac{C_{1,1}}{\lambda_1},R_{2323}=\frac{B_{12,3}^2}{\lambda_1^2}-\frac{C_1^2+C_2^2}{\lambda_1^2}+\frac{C_{2,2}}{\lambda_2};\\
&R_{1212}=\frac{C_{1,1}-C_{2,2}}{2\lambda_1}-2\frac{B_{12,3}^2}{\lambda_1^2}-\frac{C_1^2+C_2^2}{2\lambda_1^2}.
\end{split}
\end{equation}
\begin{equation}\label{ba3}
\begin{split}
&E_2(B_{12,3})=\lambda_1A_{13}-2\frac{B_{12,3}C_2}{\lambda_1},~E_1(B_{12,3})=\lambda_2A_{23}+2\frac{B_{12,3}C12}{\lambda_1};\\
&A_{12}=\frac{C_{1,2}}{\lambda_1}+\frac{B_{12,3}}{\lambda_1}\omega_{3\alpha}(E_{\alpha}),~A_{12}=\frac{-C_{2,1}}{\lambda_1}-\frac{B_{12,3}}{\lambda_1}\omega_{3\alpha}(E_{\alpha});\\
&A_{12}=\frac{-C_{1,2}}{\lambda_1}+\frac{-1}{\lambda_1}E_3(B_{12,3}),~A_{12}=\frac{C_{2,1}}{\lambda_1}+\frac{-1}{\lambda_1}E_3(B_{12,3}).
\end{split}
\end{equation}
From (\ref{ba3}), (\ref{bw3}) and (\ref{bw4}), we have
\begin{equation}\label{baw}
E_3(B_{12,3})=B_{12,3}\omega_{3\alpha}(E_{\alpha}).
\end{equation}
Define $\phi:=\omega_{3\alpha}(E_{\alpha})=\frac{E_3(B_{12,3})}{B_{12,3}}$. From (\ref{bw2}), we have
\begin{equation}\label{bwa}
\omega_{3\alpha}=\phi\omega_{\alpha}.
\end{equation}
Differentiating the equations (\ref{bwa}), we get
\begin{equation}\label{bw7}
\begin{split}
\frac{-1}{2}\sum_{kl}R_{3\alpha kl}\omega_k\wedge\omega_l&=d\phi\wedge\omega_{\alpha}-\frac{C_1}{\lambda_1}\phi\omega_1\wedge\omega_{\alpha}
-\frac{C_2}{\lambda_2}\phi\omega_2\wedge\omega_{\alpha}+\phi^2\omega_3\wedge\omega_{\alpha}\\
&-\frac{B_{12,3}C_1}{\lambda_1^2}\omega_2\wedge\omega_{\alpha}+\frac{C_1^2+C_2^2}{\lambda_1^2}\omega_3\wedge\omega_{\alpha}
-\frac{B_{12,3}C_2}{\lambda_1^2}\omega_1\wedge\omega_{\alpha}.
\end{split}
\end{equation}
Comparing the coefficients of $\omega_1\wedge\omega_{\alpha}$ and  $\omega_2\wedge\omega_{\alpha}$ on both sides of (\ref{bw7}), and using (\ref{equa4}) we obtain
\begin{equation}\label{baw1}
-A_{13}=E_1(\phi)-\frac{C_1}{\lambda_1}\phi-\frac{B_{12,3}C_2}{\lambda_1^2},-A_{23}=E_2(\phi)-\frac{C_2}{\lambda_2}\phi-\frac{B_{12,3}C_2}{\lambda_1^2}.
\end{equation}
Using $dA_{ij}+\sum_mA_{mj}\omega_{mi}+\sum_mA_{im}\omega_{mj}=\sum_mA_{ij,m}\omega_m$ and (\ref{bw2}), we obtain
\begin{equation}\label{bw6}
\begin{split}
A_{1\alpha,\alpha}=(A_{\alpha\alpha}-A_{11})\frac{C_1}{\lambda_1}+A_{12}\frac{C_2}{\lambda_1}+A_{13}\omega_{3\alpha}(E_{\alpha}),~~A_{1\alpha,k}=0,k\neq\alpha;\\
A_{2\alpha,\alpha}=(A_{22}-A_{\alpha\alpha})\frac{C_2}{\lambda_1}-A_{12}\frac{C_1}{\lambda_1}+A_{23}\omega_{3\alpha}(E_{\alpha}),~~A_{2\alpha,k}=0,k\neq\alpha;\\
A_{3\alpha,\alpha}=(A_{33}-A_{\alpha\alpha})\omega_{3\alpha}(E_{\alpha})-A_{13}\frac{C_1}{\lambda_1}+A_{23}\frac{C_2}{\lambda_1},~~A_{3\alpha,k}=0,k\neq\alpha.
\end{split}
\end{equation}
On the other hands, from (\ref{bw3}) we have
$$A_{33}-A_{\alpha\alpha}=\frac{B_{12,3}^2}{\lambda_1^2}, \alpha>3.$$
Noting that $E_{\alpha}(B_{12,3})=0$ and $A_{33,\alpha}=A_{3\alpha,3}=0$, we get
\begin{equation}\label{bwa1}
E_{\alpha}(A_{\alpha\alpha})=E_{\alpha}(A_{\beta\beta})=0,\alpha\neq\beta,\alpha,\beta>3.
\end{equation}
Combining (\ref{bw1}), (\ref{baw1}), (\ref{bw7}),(\ref{bw6}) and (\ref{bwa1}), we get
\begin{equation}\label{bwb1}
\begin{split}
A_{1\alpha,\alpha1}&=(A_{\alpha\alpha,1}-A_{11,1})\frac{C_1}{\lambda_1}+(A_{11}-A_{\alpha\alpha})[\frac{C_2^2}{\lambda_1^2}-\frac{C_{1,1}}{\lambda_1}]
+A_{12}\frac{C_1C_2}{\lambda_1^2}\\
&+A_{12,1}\frac{C_2}{\lambda_1}+A_{12}C_{2,1}+A_{13,1}\phi-A_{13}^2+\frac{A_{13}\phi C_1}{\lambda_1}-A_{12}\frac{B_{12,3}\phi}{\lambda_1};\\
A_{1\alpha,1\alpha}&=(2A_{\alpha\alpha,1}-A_{11,1})\frac{C_1}{\lambda_1}+A_{12,2}\frac{C_2}{\lambda_1}+A_{13,1}\phi.
\end{split}
\end{equation}
Combining (\ref{ba2}), (\ref{bwb1}) and Ricci identity $A_{1\alpha,1\alpha}-A_{1\alpha,\alpha1}=\sum_mA_{m\alpha}R_{m11\alpha}+\sum_mA_{1m}R_{m\alpha1\alpha}$, we obtain
\begin{equation}\label{aa1}
A_{13}=0.
\end{equation}
Similarly using Ricci identity $A_{2\alpha,2\alpha}-A_{2\alpha,\alpha2}=\sum_mA_{m\alpha}R_{m22\alpha}+\sum_mA_{2m}R_{m\alpha2\alpha}$, we have
\begin{equation}\label{aa2}
A_{23}=0.
\end{equation}
Using (\ref{aa1}), (\ref{aa2}) and $dA_{ij}+\sum_mA_{mj}\omega_{mi}+\sum_mA_{im}\omega_{mj}=\sum_mA_{ij,m}\omega_m$ and (\ref{bw2}), we obtain
\begin{equation}\label{aa3}
A_{13,2}=(A_{11}-A_{33})\frac{B_{12,3}}{\lambda_1},~A_{23,1}=(A_{22}-A_{33})\frac{B_{12,3}}{\lambda_2}.
\end{equation}
From (\ref{equa1}), we know that $A_{13,2}=A_{23,1}$, thus equations (\ref{aa3}) mean that
\begin{equation}\label{aa4}
A_{11}+A_{22}=2A_{33}.
\end{equation}
Combining (\ref{ba2}) and (\ref{aa4}), we obtain
\begin{equation}\label{aa5}
2\lambda_1^2=6\frac{B_{12,3}^2}{\lambda_1^2}-\frac{C_1^2+C_2^2}{\lambda_1^2}.
\end{equation}
Taking derivatives for (\ref{aa5}) along $E_3$ and using (\ref{equa2}) and (\ref{bc1}), we have
$$E_3({B_{12,3}})=0.$$
This means that $\phi=0$. From (\ref{baw1}), (\ref{aa1}) and (\ref{aa2}),  we get
\begin{equation}\label{bder}
B_{12,3}\frac{C_1}{\lambda_1^2}=0,~~B_{12,3}\frac{C_2}{\lambda_1^2}=0.
\end{equation}
One deduces the M\"{o}bius form $\Phi=0.$

Since $\mu=0$, we get from \eqref{equa6} that
$\lambda_1=\sqrt{\frac{n-1}{2n}},
\lambda_2=-\sqrt{\frac{n-1}{2n}}.$
Thus $f$ is a M\"{o}bius isoparametric hypersurface
with M\"{o}bius principal curvatures
\[
\sqrt{\frac{n-1}{2n}},-\sqrt{\frac{n-1}{2n}},0,\cdots,0.
\]
It is then easy to show (or by the classification result in
\cite{hu} of M\"obius isoparametric hypersurfaces with
three distinct principal curvatures) that $f$ is
 M\"{o}bius equivalent to the minimal hypersurface defined by
\[
x=(x_1,x_2):M^n=N^3\times H^{n-3}(-\frac{n-1}{6n})\rightarrow S^{n+1},
\]
where
\[
x_1=\frac{y_1}{y_0},x_2=\frac{y_2}{y_0},y_0\in R^+,
y_1\in R^5, y_2\in R^{n-3}.
\]
Here $y_1:N^3\rightarrow
S^4(\sqrt{\frac{6n}{n-1}})\hookrightarrow R^5$
is Cartan's minimal isoparametric hypersurface
in $S^4(\sqrt{\frac{6n}{n-1}})$ with
three principal curvatures, and
$(y_0,y_2):H^{n-3}(-\frac{n-1}{6n})\hookrightarrow R^{n-2}_1$
is the standard embedding of the hyperbolic space of
sectional curvature $-\frac{n-1}{6n}$ into
the $(n-2)$-dimensional Lorentz space with
$-y_0^2+y_2^2=\frac{6n}{n-1}$.

On the other hand, Since $\bar{\mu}=\mu=0$, then $B_{\alpha\alpha,1}=\bar{B}_{\alpha\alpha,1}=E_1(\bar{\mu})=0, B_{\alpha\alpha,2}=\bar{B}_{\alpha\alpha,2}=E_2(\bar{\mu})=0$. Using (\ref{equa3}), we have
\[ \bar{B}_{1\alpha,\alpha}=-\bar{C}_1,\bar{B}_{2\alpha,\alpha}=-\bar{C}_2,B_{1\alpha,\alpha}=C_1=0,B_{2\alpha,\alpha}=C_2=0.\]
For the last two of equations (\ref{cf1}), we get that
$\bar{\Phi}=0.$
Since $B,\bar{B}$ share equal eigenvalues, then $\bar{f}$ determines the same M\"obius isoparametric hypersurface as $f$.
So $f(M)$ is M\"{o}bius equivalent to $\bar{f}(M)$.

In fact, at every point of $M^n$ there exist a suitable $P=\begin{pmatrix}
\cos\theta & \sin\theta\\
-\sin\theta & \cos\theta
\end{pmatrix}\in SO(2)$
such that
\begin{equation*}\label{2b2}
\begin{pmatrix}
\bar{B}_{11}& \bar{B}_{12}\\
\bar{B}_{21}& \bar{B}_{22}
\end{pmatrix}
= P^{-1}\begin{pmatrix}
\sqrt{\frac{n-1}{2n}} & 0\\
0 & -\sqrt{\frac{n-1}{2n}}
\end{pmatrix}P.
\end{equation*}
One can show by computation that $\theta$ is a constant.
Then one can construct a local diffeomorphism
$\psi:M^n\rightarrow M^n$ such that $\bar{f}\circ\psi$
not only shares the same metric as $f$,
but also shares the same M\"obius principal curvatures
and the same principal directions. Thus there exists
M\"obius transformation $\Psi$ such that
\[\bar{f}\circ\psi=\Psi\circ f.\]
This is exactly the case as in Remark~\ref{rem3}. Thus we do not get new M\"obius deformable examples.
\end{proof}

Thus we have verified Proposition~\ref{prop-n-2} in all cases
and completed the proof to the Main Theorem~\ref{main2}.\\

{\bf Acknowledgements:} The authors thank the referees for helpful suggestions.

\end{document}